\documentclass{amsart}
\usepackage[nobysame]{amsrefs}
\usepackage{amsmath}
\usepackage{amsfonts}
\usepackage{amsthm}
\usepackage{amssymb}
\usepackage[english]{babel}
\usepackage[all]{xy}
\usepackage{color}
\usepackage{hyperref}
\usepackage{bbold}

\theoremstyle{definition}

\DeclareMathOperator{\Spec}{Spec}

\newtheorem{defi}{Definition}[section]
\newtheorem{remark}[defi]{Remark}

\newtheorem{definition}[defi]{Definition}

\theoremstyle{plain}

\newtheorem{theorem}[defi]{Theorem}
\newtheorem{corollary}[defi]{Corollary}
\newtheorem{lemma}[defi]{Lemma}
\newtheorem{proposition}[defi]{Proposition}

\newtheorem*{theorem-intro}{Theorem}

\newcommand{\Strsh}[1]{\mathcal{O}_{#1}}

\input xy
\xyoption{all}

\title{Fundamental Group-Scheme of some rationally connected fibrations}

\author{Rodrigo Codorniu Cofr\'e}
\address{Rodrigo Codorniu Cofr\'e, Universit\'e C\^ote d'Azur, Laboratoire J.A.Dieudonn\'e
UMR CNRS-UNS N${}^o$7351 \\
Universit\'e de Nice Sophia-Antipolis
Parc Valrose
06108 NICE Cedex 2 \\
e-mail:codorniu@unice.fr, racodorn@uc.cl}

\date{}
\begin{document}

\begin{abstract}
In this paper we describe the fundamental group-scheme of a proper variety fibered over an abelian variety with rationally connected fibers over an algebraically closed field. We use old and recent results for the Nori fundamental group-scheme, and of finite group-schemes in general to prove that the kernel of such a fibration is finite, and that the homotopy exact sequence holds in this case. As an application, we describe the fundamental group-schemes of certain proper varieties that are connected by curves or are related to them.
\end{abstract}  
\maketitle


\tableofcontents
\bigskip
\indent \textbf{Acknowledgements} 
This is part of my Ph.D. thesis, I would like to thank my advisors M. Antei and C. Pauly for their continued support and encouragement throughout these years. I would also like to thank M. Emsalem, F. Gounelas, F. Tonini and L. Zhang for the useful communications and exchanges related to this article. \\
This work was funded by the ANID  \footnote{ANID is the Chilean Agency of Research and Development, formerly known as CONICYT} Scholarship Program / DOCTORADO BECAS CHILE/2016 - 72170495 
\section*{Introduction} \label{sec:Intro}
In \cite{Nori76} and \cite{NoriTFGS81}, M.V. Nori developed the fundamental group-scheme, a pro-finite affine group-scheme that classifies pointed (pro-)finite torsors in terms of group-schemes, over reduced and connected schemes of finite type over a field with after fixing rational point. Moreover, if the scheme in question is proper, this group-scheme is also associated to the tannakian category of essentially finite bundles, by tannakian duality. More than 40 years later, many developments have strengthened the theory and have led to some variants of fundamental group-schemes coming from tannakian categories of vector bundles, like the S-fundamental group-scheme \cite{LangerSFGSI,LangerSFGSII} and generalizations like the fundamental gerbe by N. Borne and A. Vistoli in \cite{BorneVistoliFundGerbe}, with a recent article \cite{AntBiswEmsToniZhang2017} that extends the theory of the fundamental gerbe them even further, with strong consequences for the fundamental group-scheme that we will use here (Section \ref{subsec:Towers_of_torsors_and_FGS_of_Nori-reduced_torsors}). \\
On the other hand, there are only a few examples of schemes for which the fundamental group-scheme is known. It is known for proper normal rational schemes \cite[p. 93]{NoriTFGS81}, abelian varieties \cite{NoriFGSofAV83}, and rationally connected schemes. More precisely, normal rationally connected proper varieties have a finite fundamental group-scheme \cite{AnteiBiswasRatConnx16} and smooth proper separably rationally connected varieties have a trivial one \cite{Biswas2009SepRatCon}.\\
In his Ph.D. thesis \cite{GounelasFreeCurves}, F. Gounelas completely described elliptically connected varieties in characteristic 0: Either they are rationally connected, or they are fibered over an elliptic curve with rationally connected fibers.\\
Inspired by this result, in this article we will study the fundamental group-scheme of more general rationally connected fibrations over abelian varieties over an uncountable algebraically closed field. We have chosen abelian varieties as the base for two main reasons: the first is that the description of their fundamental group-schemes is the same regardless of dimension (Proposition \ref{prop:FGS_of_an_abelian_variety}), and thus we are not forced to just consider elliptic curves. And in second place, Nori-reduced torsors over an abelian variety are abelian varieties themselves (Remark \ref{remark:Remarks_about_the_FGS_of_an_abelian_variety}), while known examples of finite Nori-reduced torsors are not even reduced in general.\\ 
Using the fact that the fundamental group-schemes of abelian and rationally connected varieties are known as mentioned, we can state the main result of this article as follows: 
\begin{theorem-intro}[Theorem \ref{theorem:Main_theorem}]\label{teo:Main_theorem_A}
	Let $k$ be an uncountable algebraically closed field, and let $X$ be a proper variety over $k$. Assume there is a proper fibration $f:X \rightarrow S$ where $S$ is an abelian variety such that all geometric fibers are reduced, connected and possess a finite fundamental group-scheme (e.g $X$ is normal and the fibers are rationally connected). Then, there exist rational points $x \in X(k)$ and $s \in S(k)$ such that $f(x)=s$ and the following sequence of group-schemes is exact:
	$$\pi_{1}^{N}(X_{s},x) \rightarrow \pi_{1}^{N}(X,x) \rightarrow \pi_{1}^{N}(S,s) \rightarrow 1 .$$
\end{theorem-intro}
We remark that under the hypotheses of this theorem, the induced morphism between S-fundamental group-schemes is faithfully flat \cite[Lemma 8.1]{LangerSFGSI} and thus by Proposition \ref{prop:If_the_morphism_of_S-FGSs_is_faithfully_flat_then_so_is_the_morphism_between_FGSs} so is the induced morphism of fundamental group-schemes. \\
The fibration described in the main theorem is a particular case in which the homotopy exact sequence holds. To obtain the exact sequence in this case, we will use one of the equivalent conditions stated by L. Zhang in \cite{ZhangL13EGF}. We will show in the proof of Theorem \ref{theorem:Main_theorem} that one of these conditions is directly satisfied over a similar fibration $f^{\prime}:X^{\prime} \rightarrow S^{\prime}$ between Nori-reduced torsors over $X$ and $S$ respectively, satisfying the hypotheses of Lemma \ref{lemma:Homotopy_exact_sequence_when_the_kernels_descends_to_the_base}, such that we obtain the homotopy exact sequence for $f:X \rightarrow S$ out of the one obtained for $f^{\prime}$. \\
 Zhang's conditions are highly dependent on the rational points we choose on the varieties. But as $k$ is algebraically closed in our case, this choice is irrelevant, and we will show we can find a set of rational points for which the homotopy exact sequence holds, along with some simplifications to one of the conditions, called the base change condition (Definition \ref{definition:Base_change_condition}), in Section \ref{subsec:Base_change_condition}. \\
 To this purpose, we develop results and constructions, that we can roughly categorize in three groups: torsors, essentially finite bundles and finite group-schemes which are the main elements involved in the theory of the fundamental group-scheme. \\
 For torsors, stemming from the main results of \cite{AntBiswEmsToniZhang2017}, we show that projective limits of Nori-reduced torsors, that we call pro-NR torsors, possess a fundamental group-scheme (Proposition \ref{proposition:Pro-saturated_torsors_possess_a_FGS}) using the fact that finite Nori-reduced torsors over proper, reduced and connected schemes have a fundamental group-scheme and the existence of closures for towers of torsors (see Section \ref{subsec:Towers_of_torsors_and_FGS_of_Nori-reduced_torsors}). An immediate consequence of this is that the kernel of a faithfully flat morphism of fundamental group-schemes can be seen as the fundamental group-scheme of a projective limit of Nori-reduced torsors (Definition \ref{defi:Universal_pull-back_torsor} and Proposition \ref{prop:If_the_morphism_between_FGSs_is_faithfully_flat_then_the_universal_pull-back_has_a_FGS}). We also add some new terminology for the study of torsors relative to their relationship with the kernel of a morphism between fundamental group-schemes (Definition \ref{definition:Pure_pull-back_and_mixed_torsor}). And finally, we show that we can study torsors and their quotients out of their pull-backs over the generic geometric fiber (Corollary \ref{corollary:Good_correspondence_of_subgroups_implies_good_correspondence_of_quotient_torsors}) and we characterize the behavior of torsors, under the terminology of Definition \ref{definition:Pure_pull-back_and_mixed_torsor}, when pull-backed to the geometric generic fiber (Proposition \ref{prop:Pull-back_of_pure_torsors_to_the_geometric_generic_fiber} and Corollary \ref{coro:The_maximal_sub-torsor_of_a_mixed_torsor_over_the_geometric_generic_fiber_is_the_biggest_one_possible}). \\
 For essentially finite bundles, we characterize, using results from \cite{ZhangL13EGF}, their global sections in Section \ref{subsec:Global_sections_ess_fin_bundles}, and study their behavior over the geometric generic fiber of a fibration in Section \ref{subsec:Pull-back_of_pure_torsors}. We  will specifically apply these results for the essentially finite bundles coming from finite Nori-reduced torsors and in Section \ref{subsec:Base_change_condition} for the simplification of one of Zhang's equivalent conditions for the homotopy exact sequence. \\
 Finally, we show two results for finite group-schemes in Subsections \ref{subsubsec:Well_corresponded_subgroup-schemes} and \ref{subsubsec:Core_of_a_subgroup-scheme} respectively, that allows us to study torsors and their quotients when pull-backed to the geometric generic fiber. One shows that subgroups of a finite group-scheme $G$ over an algebraically closed field are in bijection (under base change) with subgroups of its base change over an algebraically closed extension of the base field (Proposition \ref{prop:Good_correspondence_for_finite_group_schemes}), and the second is the construction for finite group-schemes of an analogue for the core of a subgroup of an abstract group (Definition \ref{defi:Core_of_a_subgroup-scheme}), we show the classical properties of the core in abstract groups (Proposition \ref{prop:Properties_of_the_core_of_an_abstract_subgroup}) are also satisfied for finite group-schemes (Proposition \ref{prop:Properties_of_the_core_of_subgroup-schemes}). Both of these results are heavily used in Section \ref{subsec:Pull-back_of_pure_torsors}.\\ 
 Using the results from these three groups in conjunction, plus the fact that $S$ is an abelian variety (see Section \ref{subsec:Eta-xi_lemma}), we can show that the fibration of the main theorem has finite kernel (Proposition \ref{proposition:Finiteness_of_the_kernel}) with strong consequences (Corollary \ref{coro:Consequence_of_a_finite_kernel}) that allows us to show an specific case of a fibration for which we can directly show that one of Zhang's equivalent conditions for the homotopy exact sequence is satisfied (Lemma \ref{lemma:Homotopy_exact_sequence_when_the_kernels_descends_to_the_base}), in order to show the homotopy exact sequence holds for the fibration of the main theorem as we described in a previous paragraph. \\
As an application, we will describe the fundamental group-scheme and \'etale fundamental group of certain types of varieties connected by curves or adjacent to them in Chapter \ref{sec:Applications_to_curve_connectedness}, using some results of F. Gounelas \cite{GounelasFreeCurves}. We will show that in positive characteristic, the homotopy exact sequence holds for the varieties Gounelas described in his characterization of elliptically connected varieties in characteristic zero (Theorem \ref{teo:Rationally_connected_fibrations}), even if they might not longer be elliptically connected, as a particular case of the main theorem. Some of the results for varieties connected by curves in this chapter use the strong condition of having separably rationally connected geometric generic fibers, which have trivial fundamental group-scheme for smooth proper varieties. We show in Proposition \ref{prop:FGS_of_a_fibration_with_separably_rationally_connected_geometric_generic_fiber} a generalization, valid in positive characteristic, of a result of Kollár for the topological fundamental group for fibered complex varieties over $\mathbb{C}$ with general rationally connected fibers, see \cite[Theorem 5.2]{Kollar1993}. \\
This paper is structured as follows: In Chapter \ref{sec:Preliminaries} we define the main objects that we will work on. In Section \ref{subsec:Fundamental group-schemes} we will define and summarize the theory of the fundamental group-scheme, essentially finite and Nori-semistable bundles and the $S$-fundamental group-scheme together with some general results, plus Subsection \ref{subsubsec:Examples_of_known_FGSs} at the end to state the examples of known fundamental group-schemes we will use. In Section \ref{subsec:Global_sections_ess_fin_bundles} we will present some general useful results for global sections of essentially finite bundles. We finish this chapter with Section \ref{subsec:Towers_of_torsors_and_FGS_of_Nori-reduced_torsors}, where we summarize the consequences of \cite{AntBiswEmsToniZhang2017} for pointed finite torsors that we will use in this article. \\
In Chapter \ref{sec:FGS_of_pro-NR_torsors} we will define pro-NR torsors and show that they possess a fundamental group-scheme. \\
The main purpose of Chapter \ref{sec:Pull-back_to_the_geometric_generic_fiber} is to describe the behavior of essentially finite bundles and torsor over the geometric generic fiber in Section \ref{subsec:Pull-back_of_pure_torsors}. For the latter we need two general results for finite group-schemes, appearing in Subsections \ref{subsubsec:Well_corresponded_subgroup-schemes} and \ref{subsubsec:Core_of_a_subgroup-scheme}, that we will apply to the theory of torsors and representations of group-schemes.\\
In Chapter \ref{sec:Finiteness_of_the_kernel} we will establish the hypotheses of the main theorem for later sections (Section \ref{subsec:Setting_and_notation}), and then conceive the kernel of the induced morphism of fundamental group-schemes as the fundamental group-scheme of a pro-NR torsor under the right conditions (Section \ref{subsec:Universal_pull-back_torsor}). Then, we will apply this to our specific fibration to show that the kernel is finite in Section \ref{subsec:Proof_of_finiteness_for_the_kernel} with strong consequences. Section \ref{subsec:Eta-xi_lemma} contains a technical lemma needed for the finiteness of the kernel and the proof of the main theorem, it relies on the characterization of torsors over abelian varieties. \\
Chapter 5 contains the statement of the equivalent conditions for the homotopy exact sequence and the proof of the exact sequence for our particular fibration in Section \ref{subsec:Proof_of_main_theorem}, after defining the base change condition in Section \ref{subsec:Base_change_condition} together with some simplifications of this condition.\\
Finally, in Chapter \ref{sec:Applications_to_curve_connectedness} we will apply the results of the previous chapters to describe the fundamental group-schemes of some special cases of varieties connected by curves or associated to those. \\
Throughout this article we will often abbreviate ``fundamental group-scheme'' as FGS and ``Nori-reduced'' as NR.
\section*{Notations and conventions} \label{sec:Not_and_conv}
All group-schemes, except abelian varieties, will be affine and flat over $k$, and thus $G$-torsors $t:T \rightarrow X$ will be affine and faithfully flat over the base scheme. \\
When considering affine group-schemes as representable group-valued functors, we will do so as functors of points $\tilde{G}: \text{Alg}_{k}^{0} \rightarrow \text{Grp}$ where $\text{Alg}_{k}^{0}$ is the small category of $k$-algebras of the form $k[T_{1},\cdots,T_{n}]/I$ where $I$ is an ideal and $\{T_{i}\}_{i \in I}$ is a countable set of symbols. The full inclusion of categories $\text{Alg}_{k}^{0} \rightarrow \text{Alg}_{k}$ is an equivalence, where $\text{Alg}_{k}$ is the category of algebras of finite type and thus we will identify algebras of finite type with objects of $\text{Alg}_{k}^{0}$ using this equivalence.
\section{Preliminaries} \label{sec:Preliminaries}
\subsection{Fundamental group-schemes} \label{subsec:Fundamental group-schemes}
Let $X$ be a scheme of finite type over $k$, and let $x \in X(k)$ be a rational point. We say a $G$-torsor $t:T \rightarrow X$ is \textbf{pointed} if it has a fixed rational point $y \in T(k)$ such that $t(y)=x$. Different choices of points yield different pointed torsors. Morphisms of pointed torsors over $x$ are morphism of torsors $g:(T,t) \rightarrow (T^{\prime},t^{\prime})$ such that $g(t)=t^{\prime}$. 
\begin{defi} \label{defi:Having_a_fundamental_group-scheme}
		Let $X$ be a $k$-scheme having a rational point $x \in X(k)$. We say that $X$ \textbf{possesses a FGS} if there exists a pro-finite group-scheme $\pi_{1}^{N}(X,x)$ over $k$ and a pointed $\pi_{1}^{N}(X,x)$-torsor, denoted as $\hat{X} \rightarrow X$ and called \textbf{universal torsor}, that are unique up to isomorphism. The torsor $\hat{X}$ is universal in the sense that there exists a unique morphism of torsors $\hat{X} \rightarrow T$ for any pointed (pro-)finite torsor $T \rightarrow X$. \\
		Equivalently, $X$ possesses a FGS if there exists a bijection of sets
		$$ \left\{ t:T \rightarrow X: \, T \text{ is a pointed $G$-torsor} \right\} \overset{\sim}{\rightarrow} \left\{ \text{Arrows $\pi_{1}^{N}(X,x) \rightarrow G$ }  \right\} $$
		that is natural on $G$, for any (pro-)finite group-scheme $G$ over $k$. \\
		The bijection is given by taking a $G$-torsor $T \rightarrow X$, with unique morphism $\hat{X} \rightarrow T$, to the induced morphism $\pi_{1}^{N}(X,x) \rightarrow G$ of fibers over $x$, whose inverse consists of taking such an arrow of group-schemes and considering the contracted product torsor $\hat{X} \times^{\pi_{1}^{N}(X,x)} G$ along this morphism, which is pointed $G$-torsor.
\end{defi}
We will further describe the FGS and $\hat{X}$, but first, we need to define a special type of pointed torsors.
\begin{definition} \label{definition:Nori-reduced_torsor}
	Let $X$ be a $k$-scheme with a rational point $x \in X(k)$. A pointed torsor $t:T \rightarrow X$ is \textbf{Nori-reduced} if it does not possess any non-trivial pointed sub-torsor or equivalently, if any morphism of pointed torsors $T^{\prime} \rightarrow T$ over $X$ is faithfully flat.
\end{definition}
\begin{remark} \label{remark:Nori_reduced_torsors_in_terms_of_the_FGS}
	If $X$ possesses a FGS. Then a pointed $G$-torsor over $X$ is Nori-reduced if and only if the arrow $\pi_{1}^{N}(X,x) \rightarrow G$ is faithfully flat.
\end{remark}
In \cite{NoriTFGS81}, Nori showed that if $X$ is connected and reduced, it possesses a FGS as the category $\text{Tors}_{X,x}$ of pointed pro-finite torsors over $X$ is co-filtered This implies that $\hat{X}$ is the co-filtered limit of this category and moreover, that as a pro-finite group-scheme, the finite quotients of $\pi_{1}^{N}(X,x)$ are finite group-schemes corresponding to Nori-reduced torsors. \\
An important construction that we will use throughout this article is the following: Let $t:T \rightarrow X$ be a pointed $G$-torsor with $G$ finite. Let us suppose that the image of $\pi_{1}^{N}(X,x) \rightarrow G$ is a proper subgroup-scheme $H \subset G$, then there exists a Nori-reduced pointed $H$-torsor $V \rightarrow X$ which is the smallest sub-torsor of $T$, it is a closed sub-scheme of $T$. \\
From now on, we will assume that $X$ is proper and $k$ is perfect, in this case we have a richer description of the FGS in terms of its representations and a certain family of vector bundles over $X$.
\begin{definition} \label{defi:Finite_and_essentially_finite_bundles}
	Let $\mathcal{F}$ be a vector bundle over $X$. We say that $\mathcal{F}$ is \textbf{finite} if there exist two different polynomials $f,g \in \mathbb{Z}_{\geq 0}[x]$ such that $f(\mathcal{F}) \cong g(\mathcal{F})$ where for a polynomial $p(x)=a_{0}+a_{1}x+a_{2}x^{2} + \cdots + a_{n}x^{n}$ we define
	$$p(\mathcal{F}) = \Strsh{X}^{\oplus a_{0}} \oplus \mathcal{F}^{\oplus a_{1}} \oplus \left( \mathcal{F}^{\otimes 2} \right)^{\oplus a_{2}} \oplus \cdots \oplus \left( \mathcal{F}^{\otimes n} \right)^{\oplus a_{n}}.$$
	A vector bundle $\mathcal{E}$ is \textbf{essentially finite} if it is the kernel of a morphism of bundles $\varphi:\mathcal{F} \rightarrow \mathcal{G}$ where $\mathcal{F}$ and $\mathcal{G}$ are finite bundles.
\end{definition}
\begin{remark} \label{remark:Alternative_definition_essentially_finite_bundles}
	The definition of essentially finite bundles is due to N. Borne and A. Vistoli in \cite[Def. 7.7]{BorneVistoliFundGerbe}. The more classical definition is that $\mathcal{F}$ is essentially finite if it is Nori-semistable (see the definition below) and it is a ``sub-quotient of a finite bundle'', this means that there exists a finite bundle $\mathcal{F}$ and two sub-bundles $V^{\prime} \subset V \subset \mathcal{F}$ such that $\mathcal{E} = V/V^{\prime}$.
\end{remark}
We will also need the following family of vector bundles for a later proposition:
\begin{definition} \label{defi:Nori-semistable_bundles}
	A vector bundle $\mathcal{F}$ over $X$ is \textbf{Nori-semistable} if for any non-constant morphism $f:C \rightarrow X$ from a smooth and projective curve $C$, the pull-back bundle $f^{*}(\mathcal{F})$ is semi-stable of degree 0.
\end{definition}
\begin{remark} \label{remark:Remark_about_Nori-semistable_bundles}
	The definition we outline in this article is not standard, we are using a definition that allows us to utilize the theory of the S-fundamental group-scheme, developed first by Biswas, Parameswaran and Subramanian for curves \cite{BiswasParameswaranSubramanianSSGFCurv}, and later for more general schemes in \cite{LangerSFGSI,LangerSFGSII} by A. Langer. The terminology used in those articles is of ``numerically flat'' bundles. Another approach for this fundamental group-scheme that works for a wider class of schemes is outlined in \cite[\S 7]{BiswasHaiDosSantos2018Cociente}. Both approaches are stated for algebraically closed fields, Langer's approach works also when $k$ is perfect.
\end{remark} 
Let us denote as $\text{EF}(X)$ and $\text{NS}(X)$ the categories of essentially finite and Nori-semistable bundles respectively. If $\mathcal{F}$ is  finite bundle over $X$, then for $f:C \rightarrow X$ a morphism from a smooth projective curve, we have that $f^{*}(\mathcal{F})$ is finite and thus semi-stable of degree 0. As the category of semi-stable bundles of fixed slope is abelian (see \cite[Prop. 9]{Sesh82}), we see that the category $\text{EF}(X)$ is fully included in the category $\text{NS}(X)$. These categories are both special and are related to group-schemes in the following way:
\begin{definition} \label{defi:Tannakian_category}
	A \textbf{neutral tannakian category} over a field $k$, is a category $\mathcal{C}$ that is $k$-linear, abelian, rigid and a tensor category with $\text{End}(\mathbb{1})=k$, together with an additive tensor functor $\omega:\mathcal{C} \rightarrow \text{Vectf}_{k}$ to the category of finitely-dimensional $k$-vector spaces, called the \textbf{fiber functor}, that is is exact and faithful.  
\end{definition}
The quintessential neutral tannakian category is the category $\text{Rep}_{k}(G)$ of finitely-dimensional representations of an affine group-scheme $G$. If $G=\Spec(A)$ where $A$ is a Hopf algebra and $V$ is a finitely-dimensional $k$-vector space, for a \textbf{representation} of $G$ over $V$ we will mean indistinctly a comodule $\rho:V \rightarrow V \otimes_{k} A$ or a morphism of group-schemes $G \rightarrow GL(V)$ that in the level of functors of points, corresponds to morphisms of groups $\tilde{G}(R) \rightarrow \text{Aut}_{R}(V \otimes_{k} R)$. In this case the fiber functor $\omega_{G}$ that we associate to $\text{Rep}_{k}(G)$ is the forgetful functor that assigns to a representation, the underlying vector space $V$. \\
By a well-known result, called \emph{Tannakian correspondence}, any neutral tannakian category is equivalent to the category of representations of an affine group-scheme. Under this correspondence, morphism of group-schemes $f:G \rightarrow G^{\prime}$ correspond to tensor functors between tannakian categories $F:\text{Rep}_{k}(G^{\prime}) \rightarrow \text{Rep}_{k}(G)$ such that $\omega_{G^{\prime}}=\omega_{G} \circ F$. \\
 The tannakian correspondence pertains to $\text{EF}(X)$ and $\text{NS}(X)$, and in the next proposition we will outline the main properties of their corresponding group-schemes and the properties of them that we will use in this article:
\begin{proposition} \label{prop:Properties_of_the_Nori_and_S-fundamental_group-schemes}
	Let $X$ be a reduced, connected and proper scheme over $k$ with $x \in X(k)$. Then, the categories $\text{EF}(X)$ and $\text{NS}(X)$ are neutral tannakian with the fiber functor $\omega_{x}$ that assigns to a bundle its fiber over $x$. The group-schemes associated to $\text{EF}(X)$ and $\text{NS}(X)$ are $\pi_{1}^{N}(X,x)$ and $\pi_{1}^{S}(X,x)$ respectively and the following properties hold:
	\begin{enumerate}
		\item There is a natural faithfully flat morphism of group-schemes $\pi_{1}^{S}(X,x) \rightarrow \pi_{1}^{N}(X,x)$ corresponding from the full inclusion of categories $\text{EF}(X) \rightarrow \text{NS}(X)$. 
		\item Any finite quotient $\pi_{1}^{S}(X,x) \rightarrow G$ factors through $\pi_{1}^{N}(X,x)$.
	\end{enumerate}
\end{proposition}
The group-scheme $\pi_{1}^{S}(X,x)$ is known as the \textbf{S-fundamental group-scheme}.
\proof
For the proof that both $\text{EF}(X)$ and $\text{NS}(X)$ are tannakian, see \cite[Ch. 1]{NoriTFGS81} and \cite[Theorem 7.8]{BiswasHaiDosSantos2018Cociente} respectively. The proofs of (1) and (2) can be found in \cite[\S 6]{LangerSFGSI}. \\
\endproof
Now we can prove the following lemma:
\begin{lemma} \label{lemma:If_the_pull-back_of_a_bundle_is_ess_finite_then_it_is_ess_finite}
	Let $f:X \rightarrow S$ be a faithfully flat morphism between proper, reduced and connected $k$-schemes with $f_{*}(\Strsh{X})=\Strsh{S}$ and such that $f(x)=s$ for some $x \in X(k)$ and $s \in S(k)$. Let $V \in \text{Vect}(S)$ be a vector bundle such that $f^{*}(V)$ is essentially finite, then $V$ is essentially finite.
\end{lemma}
\proof
First we notice that from \cite[Lemma 8.1]{LangerSFGSI} the induced morphism between S-fundamental group-schemes $\pi^{S}(f): \pi_{1}^{S}(X,x) \rightarrow \pi_{1}^{S}(S,s)$ is faithfully flat. This implies that the pull-back functor $f^{*}:NS(S) \rightarrow NS(X)$ is fully faithful and the essential image of this functor is ``closed by sub-objects'' (see \cite[2.21 (a)]{DeligneMilneTann}), which means that for any $W \in NS(S)$ and a sub-object of $U^{\prime} \subset f^{*}(W)$, there exists a sub-object $U \subset W$ such that $f^{*}(U) \cong U^{\prime}$. \\
Now let us prove that $V$ is Nori-semistable. For this purpose, let $v:C \rightarrow S$ be a non-constant morphism from a proper and smooth curve, if we take the fibered product $ C \times_{S} X$ this is a reduced, proper and connected scheme and we can always consider a morphism $C^{\prime} \rightarrow C \times_{S} X$ with $C^{\prime}$ a proper and smooth curve that passes through any pair of arbitrary points using \cite[Lemma p.56]{MumfAV70}. This gives us a morphism $w:C^{\prime} \rightarrow X$ making the following diagram commutative
$$ \xymatrix{C^{\prime} \ar[r]^{c} \ar[d]_{w}& C \ar[d]^{v}\\
             X \ar[r]_{f}&S} $$
and we can moreover chose the points on the fibered product such that $c:C^{\prime} \rightarrow C$ is finite, surjective and $w:C^{\prime} \rightarrow X$ is non-constant. Now we have that $w^{*}\left(f^{*}(V)\right) = c^{*}\left(v^{*}(V)\right)$ and as $f^{*}(V)$ is essentially finite, the left hand side of the equation is semi-stable of degree 0, then so is $v^{*}(V)$. Thus, $V$ is Nori-semistable. \\
Finally, to prove that $V$ is essentially finite, it suffices to prove that $V$ is contained in a tannakian full subcategory of $NS(S)$ associated to a finite group-scheme, as from \cite[Ch. I \S2]{NoriTFGS81}, essentially finite bundles are precisely those that become trivial after taking pull-backs to finite torsors and any vector bundle contained in a tannakian full subcategory $\text{Rep}_{k}(G)$ of coherent sheaves with $G$ affine becomes trivial after taking the pull-back along a $G$-torsor $T \rightarrow X$ associated to the category.\\
 Let $\left\langle f^{*}(V)\right\rangle^{\otimes}$ be the full subcategory of $NS(X)$ composed of bundles isomorphic to a sub-quotient of finite direct sums of copies of $f^{*}(V)$. As $f^{*}(V)$ is essentially finite, $\left\langle f^{*}(V)\right\rangle^{\otimes}$ is fully contained in a tannakian full subcategory $\text{Rep}_{k}(H)$ of Nori-semistable bundles with $H$ finite. Let $\mathcal{C}$ be the full subcategory of $NS(S)$ of bundles whose pull-backs belong to $\text{Rep}_{k}(H)$, clearly $V$ is an object of $\mathcal{C}$ and we have a restricted functor $f^{*}:\mathcal{C} \rightarrow \text{Rep}_{k}(H)$. From \cite[2.20 (a)]{DeligneMilneTann}, $\text{Rep}_{k}(H) \cong \left\langle W \right\rangle^{\otimes}$ with $W$ an essentially finite bundle, and thus we can easily see that $f^{*}$ is tensorial, exact, fully faithful and essentially surjective using the observation of the first paragraph, thus $\mathcal{C}$ and $\text{Rep}_{k}(H)$ are equivalent and $V$ is essentially finite.
\endproof
\begin{remark} \label{remark:Simplification_of_pull-back_of_an_essentially_finite_in_the_normal_case}
	If $X$ is a normal variety, we can characterize essentially finite bundles as bundles $V$ over $X$ as those that become trivial after taking the pull-back along a surjective morphism $g:Y \rightarrow X$ such that $g_{*}(\Strsh{Y})$ is coherent, see \cite{ToniniZhangTrivialized}. In this case, Lemma \ref{lemma:If_the_pull-back_of_a_bundle_is_ess_finite_then_it_is_ess_finite} is trivial and we only require that $f$ is proper. 
\end{remark}
The proof of Lemma \ref{lemma:If_the_pull-back_of_a_bundle_is_ess_finite_then_it_is_ess_finite} also yields the following:
\begin{proposition} \label{prop:If_the_morphism_of_S-FGSs_is_faithfully_flat_then_so_is_the_morphism_between_FGSs}
	Let $f:X \rightarrow S$ be a morphism between proper, reduced and connected schemes over $k$. Let us suppose that the induced morphism $\pi_{1}^{S}(X,x) \rightarrow \pi_{1}^{S}(X,x)$ is faithfully flat for rational points $x\in X(k)$ and $s \in S(k)$ satisfying $f(x)=s$ (e.g. assuming the hypotheses of Lemma \ref{lemma:If_the_pull-back_of_a_bundle_is_ess_finite_then_it_is_ess_finite}), then so is the induced morphism $\pi_{1}^{N}(X,x) \rightarrow \pi_{1}^{N}(S,s)$.
\end{proposition}
\proof
Let $V$ be an essentially finite bundle over $S$, we can consider the smallest tannakian full sub-category $\text{Rep}_{k}(G)$ of $\text{EF}(S)$ containing $V$, with $G$ finite. If $\text{Rep}_{k}(H)$ is the same tannakian full sub-category of $\text{EF}(X)$ applied to $f^{*}(V)$, we see from the proof of Lemma \ref{lemma:If_the_pull-back_of_a_bundle_is_ess_finite_then_it_is_ess_finite} that they are equivalent. If we translate this into group-schemes, we have a commutative square
$$ \xymatrix{\pi_{1}^{S}(X,x) \ar[d] \ar[r]&\pi_{1}^{S}(S,s) \ar[d] \\
             H \ar[r] &G}$$ 
where the upper horizontal arrow and both vertical arrows are faithfully flat, and the lower horizontal arrow is an isomorphism. Then, by Proposition \ref{prop:Properties_of_the_Nori_and_S-fundamental_group-schemes} (2) we conclude the proof.
\endproof
\subsubsection{Examples of known fundamental group-schemes} \label{subsubsec:Examples_of_known_FGSs}
Here we will state the results for the known examples of fundamental group-schemes we will use in this article. Starting with abelian varieties, we have the following:
\begin{proposition}[\cite{NoriFGSofAV83}] \label{prop:FGS_of_an_abelian_variety}
	Let $S$ be an abelian variety over a perfect field $k$. For $n \geq 1$, let $S[n]$ be the kernel of the multiplication morphism $m_{n}:S \rightarrow S$, which is a Nori-reduced $S[n]$-torsor. Then 
	$$\pi_{1}^{N}(S,0)= \underset{n}{\lim_{\leftarrow}} S[n].$$
\end{proposition}
\begin{remark} \label{remark:Remarks_about_the_FGS_of_an_abelian_variety}
	In \cite{NoriFGSofAV83}, Nori showed that if $t:T \rightarrow S$ is a pointed Nori-reduced $G$-torsor over $0$, then there exists a unique integer $N \geq 1$ an a morphism of torsors $p:S \rightarrow T$ such that $t \circ p = m_{N}$, and thus we have a faithfully flat morphism of group schemes $S[N] \rightarrow G$. At the time Nori published this result, the K\"{u}nneth formula was not yet established for the FGS but he conjectured the limit formula for $\pi_{1}^{N}(S,0)$ from this result, assuming the formula. Later, Mehta and Subramanian showed the K\"{u}nneth formula in \cite{MethaSubramanianEGF}. \\
	The morphism $p:S \rightarrow T$ implies that $T$, as the contracted product of an abelian variety, is an abelian variety itself. When $k$ is algebraically closed, this holds even if $T$ is not pointed over $0$: if $T$ is pointed over $s \neq 0$, by post-composing $t$ by a translation, we see that $T$ can be considered to be pointed over $0$. Thus, we will get a morphism of torsors $p:S \rightarrow T$ over $0$ that can be made into a morphism of torsors over $s$ again by composing with another translation afterwards.
\end{remark}
Now let us consider the other family of schemes with known FGS that we will use:
\begin{defi}[Definition 3.2 \S IV.3 \cite{KollarBookRat}] \label{defi:Rationally_connected_varieties}
	Let $X$ be a variety over $k$, we will not assume $X$ is proper. We say for that $X$ is \textbf{rationally connected (resp. rationally chain connected)}, if there exist a proper and flat family of curves $\mathcal{C} \rightarrow Y$ where $Y$ is a variety, whose geometric fibers are proper smooth irreducible rational curves (resp. connected curves with smooth irreducible components that rational curves), such that there exists a morphism $u:\mathcal{C} \rightarrow X$, making $u^{(2)}:\mathcal{C} \times_{Y} \mathcal{C} \rightarrow X \times_{k} X$ dominant. \\
	Moreover, if $X$ is rationally connected (resp. rationally chain connected) and $u^{(2)}$ is smooth at the generic point, we say that $X$ is \textbf{separably rationally connected (resp. rationally chain connected)}.
\end{defi} 
\begin{remark} \label{remark:Remarks_rationally_connected_varieties}
	If $\text{char}(k)=0$, by generic smoothness, a rationally connected (resp. rationally chain connected) variety is also separably rationally connected (resp. separably rationally chain connected), thus these notions coincide. In positive characteristic, there are examples of rationally connected varieties that are not separably rationally connected, see \cite[Ch. V 5.19]{KollarBookRat}.   
\end{remark}
In terms of the FGS, we have the following result:
\begin{proposition} \label{prop:FGS_of_rationally_connected_varieties}
	Let $X$ be a proper and normal variety over an algebraically closed field $k$. If $X$ is rationally chain connected, then $\pi_{1}^{N}(X)$ is finite. \\
	If in addition, $X$ is separably rationally connected and smooth, then its FGS is trivial.
\end{proposition}
\proof 
For the first assertion, see \cite{AnteiBiswasRatConnx16}. The second one can be found in \cite{Biswas2009SepRatCon}.
\endproof
\subsection{Global sections of essentially finite bundles} \label{subsec:Global_sections_ess_fin_bundles}
Let $X$ be a proper, reduced and connected scheme of finite type over a perfect field $k$ with $x \in X(k)$. We will describe the relationship between global sections of essentially finite bundles and their properties. We start with the following definition:
\begin{defi} \label{defi:Order_of_a_finite_scheme}
	Let $Y$ be a finite scheme over $k$. This implies that $Y=\Spec(A)$ is affine where $A$ is a $k$-algebra that is finitely dimensional as a $k$-vector space. In this case we define the \textbf{order of} $Y$ as $\text{ord}(Y):=\dim_{k}(A)$.
\end{defi}
Let $V$ be an essentially finite bundle over $X$, we recall the following result of L. Zhang:
\begin{lemma} [Lemma 2.2 \cite{ZhangL13EGF}] \label{lemma:Zhangs_lemma}
	The natural morphism $\Gamma(X,V) \otimes_{k} \Strsh{X} \rightarrow V$ is an embedding that makes $\Gamma(X,V) \otimes \Strsh{X}$ the maximal trivial sub-bundle of $V$.
\end{lemma}
\begin{remark} \label{remark:Trivial_ess_finite_bundles_in_terms_of_global_sections}
	An immediate consequence of the lemma is that an essentially finite bundle $V$ of rank $r$ is trivial if and only if $\dim_{k}(\Gamma(X,V))=r$. Also we have in general that $0 \leq \dim_{k}(\Gamma(X,V)) \leq r$.
\end{remark}
Now let $t:T \rightarrow X$ be a $G$-torsor with $G$ finite. In this case $V_{T}:=t_{*}(\Strsh{T})$ is a finite bundle over $X$ satisfying $V_{T} \otimes V_{T} \cong V_{T}^{\oplus n}$ where $n=\text{ord}(G)$. It is a well-known fact that $T$ is Nori-reduced if and only if $\Gamma(X,V_{T})=k$ (\cite[II Prop. 3]{NoriTFGS81}). \\
Excluding the trivial and Nori-reduced cases, we can characterize the global sections of $V_{T}$. But first, we need some terminology: Let $H$ be the image of $\pi_{1}^{N}(X,x) \rightarrow G$, it is a proper subgroup-scheme of $G$. If $G=\Spec(A)$ where $A$ is a Hopf algebra over $k$, we will write $H=\Spec(A/I)$ where $I$ is a Hopf ideal. Let $\rho_{G}:A \rightarrow A \otimes A$ be the regular representation of $G$, we must point out that it is not only a $k$-linear morphism but also a $k$-algebra morphism. If we restrict this representation using the projection $A \rightarrow A/I$, we obtain the representation $\tau:A \rightarrow A \otimes A/I$ corresponding to the action coming from restricting the multiplication morphism $m:G \times_{k} G \rightarrow G$ of $G$ to $H$ on one coordinate. The categorical quotient of this latter action is a finite scheme $G/H$ which corresponds to the sub-algebra $A^{H}:=\{a \in A: \, \tau(a) = a \otimes \bar{1} \}$ of elements that are trivially acted upon, it is the maximal sub-space of $A$ on which the $\tau$ acts trivially. If $H$ is normal, then $G/H$ is moreover a group-scheme.
\begin{corollary} \label{coro:Global_sections_of_bundles_coming_from_torsors}
	Let $t:T \rightarrow X$ be a $G$-torsor with $G$ finite that is neither Nori-reduced nor trivial. Then, we have that $\dim_{k}(\Gamma(X,V_{T}))= \text{ord}(G/H)$.
\end{corollary}
\proof
Let $F_{G}:\text{Rep}_{k}(G) \rightarrow \text{Rep}_{k}(\pi_{1}^{N}(X,x))$ be the functor coming from the morphism of group-schemes $\pi_{1}^{N}(X,x) \rightarrow G$, it corresponds in terms of essentially finite bundles to the functor $W \mapsto \left(W \otimes_{k} V_{T}\right)^{G}$ for $W \in  \text{Rep}_{k}(G)$. From this we see that $V_{T}$ is the image of $\rho_{G}$ via this functor, but from the factorization $ \text{Rep}_{k}(G) \rightarrow \text{Rep}_{k}(H) \rightarrow \text{Rep}_{k}(\pi_{1}^{N}(X,x))$ we see that $\tau \in \text{Rep}_{k}(H)$ maps to $V_{T}$ as well. \\
Now let $U \subset V_{T}$ be the maximal trivial sub-bundle, as $\pi_{1}^{N}(X,x) \rightarrow H$ is faithfully flat, using \cite[2.21 (a)]{DeligneMilneTann} we can find a trivial sub-representation of $\tau$ whose image in $\text{Rep}_{k}(\pi_{1}^{N}(X,x))$ is isomorphic to $U$. We can easily see that this the maximal sub-representation of $\tau$ with a trivial action and thus it corresponds to $A^{H}$, finishing the proof.
\endproof
\begin{remark} \label{remark::Essentially_finite_bundle_associated_to_a_quotient_of_a_torsor}
	Let $t^{\prime}:T/H \rightarrow X$ be the canonical morphism from the quotient of $T$ by the restricted $H$-action. Then we can easily see that $\left(t^{\prime}\right)_{*}(\Strsh{T/H})$ is a sub-bundle of $V_{T}$ that corresponds to the sub-representation $A^{H} \subset A$ of $H$. We will expand this further in Subsection \ref{subsubsec:Core_of_a_subgroup-scheme}.
\end{remark}
\subsection{Tower of Torsors and the FGS of a Nori-reduced torsor} \label{subsec:Towers_of_torsors_and_FGS_of_Nori-reduced_torsors}
In this section, we will outline all the recent results developed in \cite{AntBiswEmsToniZhang2017} that we will use in this paper. The general context of this article is to show the existence of ``Galois closures'' for a broader family of covers, containing the \'etale covers, over schemes, and more generally, over algebraic stacks over a field. Although the general results are stated for algebraic stacks and the ``fundamental gerbe'', a generalization of the fundamental group-scheme for fibered categories developed by N. Borne and A. Vistoli in \cite{BorneVistoliFundGerbe}, we will only use them in the context of $k$-schemes of finite type. \\
More specifically, we will work over a perfect field $k$, and all schemes considered will be of finite type, proper, reduced and connected with a rational point. To contextualize the hypotheses of the results we will state in this section, with those used in \cite{AntBiswEmsToniZhang2017} for the theory of the fundamental gerbe, ``inflexible'' means having a FGS and proper implies the ``pseudo-proper'' property, though it is not equivalent. The general notion of cover is the following:
\begin{defi} \label{defi:Essentially_finite_cover}
	Let $X$ be a $k$-scheme. An \textbf{essentially finite cover} is a $k$-scheme $Y$ together with finite and faithfully flat morphism $f:Y \rightarrow X$ such that $f_{*}(\Strsh{Y})$ is an essentially finite bundle over $X$. If $x \in X(k)$ and there is $y \in Y(k)$ such that $f(y)=x$ we will say in addition that the essentially finite cover is \textbf{pointed}
\end{defi}
Examples of essentially finite covers include \'etale covers, $G$-torsors $T \rightarrow X$ with $G$ finite and their quotients by the restricted action over any subgroup-scheme $T/H \rightarrow X$ where $H \subset G$. We will not directly use Galois closure of towers. For a definition and properties of Galois closures, see \cite[Theorem II]{AntBiswEmsToniZhang2017}.\\
The first result we will use is the following:
\begin{proposition}[Corollary I \cite{AntBiswEmsToniZhang2017}] \label{prop:FGS_of_Nori-reduced_torsors}
	Let $X$ be a connected, reduced and proper $k$-scheme of finite type that possesses a FGS with $x \in X(k)$. Then a pointed and finite $G$-torsor $T \rightarrow X$ possesses a FGS if and only if it is Nori-reduced and in that case, for $t \in T(k)$ over $x$, we have an exact sequence
	$$ 1 \rightarrow \pi_{1}^{N}(T,t) \rightarrow \pi_{1}^{N}(X,x) \rightarrow G \rightarrow 1 .$$
\end{proposition}
For the second result we that need, we introduce the following definition:
\begin{definition}[Definition 3.8 \cite{AntBiswEmsToniZhang2017}] \label{definition:Envelopes_and_closure_of_a_tower_of_torsors}
	Let $X$ be a scheme over a field $k$ and let $Z \rightarrow Y \rightarrow X$ be a \textbf{tower of finite pointed torsors}. This means that $Z \rightarrow Y$ and $Y \rightarrow X$ are finite pointed torsors. If $G$ and $H$ are the finite group-schemes associated to $Z \rightarrow Y$ and $Y \rightarrow X$ respectively, we say that a $K$-torsor $U \rightarrow X$ is an \textbf{envelope} of or that it \textbf{envelops} the tower $Z \rightarrow Y \rightarrow X$ if we have morphisms of group-schemes $\alpha:K \rightarrow G$ and $\beta:\ker(\alpha) \rightarrow H$, and a morphism $U \rightarrow Z$ making the following diagram commutative
	\begin{equation} \label{equation:Closure_of_a_tower}
	\xymatrix{ & U \ar[rd] \ar[ld] \ar[d] & \\
		Z \ar[r] & Y  \ar[r]& X}
	\end{equation}
	so that we have a morphism of torsors $U \rightarrow Y$ over $T$ intertwining the respective group-scheme actions via $\alpha$, and a morphism of schemes $U \rightarrow Z$ over $W$ that intertwines the corresponding actions via $\beta$. An envelope is \textbf{Nori-reduced} if $U \rightarrow X$ is Nori-reduced. \\
	Finally, if an envelope $U$ is minimal in the sense that every other envelope $U^{\prime}$ possesses a canonical arrow $U^{\prime} \rightarrow U$ that is a morphism of torsors over $X$, we say that $U$ is the \textbf{closure} of the tower $Z \rightarrow Y \rightarrow X$. Closures are unique up to isomorphism.
\end{definition}
Now we can state the second result:
\begin{proposition}[Theorem III \cite{AntBiswEmsToniZhang2017}] \label{prop:Closure_of_towers_for_schemes_having_a_FGS}
	Let $X$ be a connected, reduced and proper $k$-scheme of finite type that possesses a FGS with $x \in X(k)$. Then any tower of torsors over $X$ possesses a Nori-reduced closure. Moreover, if $Z \rightarrow Y \rightarrow X$ is a tower of torsors over $X$ and $U \rightarrow X$ is its closure, we have that the morphism $U \rightarrow Z$ is faithfully flat if and only if both members of the tower are Nori-reduced over its respective bases, or they both posses a FGS in an equivalent fashion, and in that case $U$ is a Nori-reduced torsor over $Z$ and $Y$.
\end{proposition}
\begin{remark} \label{remark:About_towers_of_torsors}
	For a $k$-scheme $X$ satisfying the hypotheses of Proposition \ref{prop:Closure_of_towers_for_schemes_having_a_FGS}, the composition $Z \rightarrow X$ is an essentially finite cover and the closure of any tower coincides with the Galois closure of this cover, see \cite[Theorem 3.10]{AntBiswEmsToniZhang2017} for this result.
\end{remark}
\begin{remark} \label{remark:A_scheme_possessing_a_FGS_shares_its_universal_torsor_with_its_NR_torsors}
	Using Proposition \ref{prop:Closure_of_towers_for_schemes_having_a_FGS}, we can see that the universal torsor of a Nori-reduced torsor $T \rightarrow X$ is $\hat{X}$, the universal torsor of $X$.
\end{remark}
\section{FGS of pro-NR torsors} \label{sec:FGS_of_pro-NR_torsors}
\subsection{Notation for projective limits of torsors} \label{subsec:Notation_for_projective_limits_of_torsors}
In this chapter and later, we will work with projective limits of torsors over a $k$-scheme $X$, for the moment $k$ will be a general field. \\
We will fix notation for this first: if $\{T_{i}\}_{i \in I}$ is an inverse limit of finite torsors $T_{i} \rightarrow X$ over a partially ordered set $I$. The limit of this system will be denoted by $T := \displaystyle{\lim_{\leftarrow}} \, T_{i}$. We will also consider the associated inverse system of group-schemes $\{G_{i}\}_{i \in I}$, being $G_{i}$ the group-scheme associated to $T_{i}$. \\ 
Finally, for the pointed case, if $x \in X(k)$, the points $t_{i}$ and $t$ will denote respectively a rational point over $T_{i}$ and $T$ ($i \in I$), clearly $t$ is the inverse limit of the system formed by the $t_{i}$. When	 needed, we may add an index $0$ to the set $I$ such that $T_{0} := X$ and $t_{0} := x$.
\subsection{Pro-NR torsors} \label{subsec:Pro-NR_torsors}
We are interested in a particular type of limit of torsors:
\begin{definition} \label{definition:Pro-NR_torsor}
	Let $\{T_{i}\}_{i \in I}$ be a co-filtered family of finite pointed Nori-reduced torsors over a $k$-scheme $X$. We will call its projective limit $T:=\displaystyle{\lim_{\leftarrow}} \, T_{i}$ a \textbf{pro-NR torsor}. In the case we have compatible rational points over any member of the limit, according to our notation, we will say that this pro-NR torsor is \textbf{pointed}.
\end{definition}
We want to show that if $X$ is of finite type, proper, reduced and connected, any pro-NR torsor possesses a FGS with the same property in terms of its FGS as in the case of Nori-reduced torsors (Proposition \ref{prop:FGS_of_Nori-reduced_torsors}). For this, we need a lemma that allows us to handle torsors over the limit scheme $T$.
\begin{lemma} \label{lemma:Torsors_over_pro-NR_torsors_descent_to_a_finite_torsor}
	Let $T:=\displaystyle{\lim_{\leftarrow}} T_{i}$ be a projective limit of finite torsors over a scheme $X$ quasi-compact and quasi-separated over a field $k$. Let $V,W \rightarrow T$ be two finite torsors over $T$, where $G$ and $H$ are their corresponding group-schemes respectively. Then:
	\begin{enumerate}
		\item There exist an index $i \in I$ and a finite $G$-torsor $V_{i} \rightarrow T_{i}$ such that the following diagram is cartesian
		$$ \xymatrix{V \ar[d] \ar[r] & V_{i} \ar[d] \\
			T \ar[r] & T_{i} }.$$
		In addition, if $V$ is Nori-reduced, then so is $V_{i}$.
		\item If $\phi: V \rightarrow W$ is a morphism of torsors over $T$, there exists an index $i$ such that $V_{i}$ and $W_{i}$ are torsors over $T_{i}$ with a morphism of torsors $\phi_{i}: V_{i} \rightarrow W_{i}$ over $T_{i}$ such that $\phi$ is the pull-back of $\phi_{i}$ over $T$. In that case, $\phi$ is a isomorphism if and only if for an index $j \geq i$ the pull-back of $\phi_{i}$ to $T_{j}$ is an isomorphism.
	\end{enumerate}
\end{lemma}
We recall that a morphism of schemes $g:X \rightarrow S$ is \textbf{quasi-compact} if the inverse image of every affine open of $S$ is quasi-compact and we say that $X$ is \textbf{quasi-separated} over $S$ if the diagonal morphism $\Delta_{X/S}:X \rightarrow X \times_{S} X$ is quasi-compact.
\proof
The proof mainly relies on \cite[\S 8.8]{EGA4iii} where basic properties of schemes (locally of finite type) over a projective limit of schemes with affine transition maps, and morphisms between those are stated.\\
For part (1), if we apply \cite[8.8.2 (ii)]{EGA4iii} to $V$, which is fppf over $T$, there exists an index $i \in I$ and a fppf scheme $V_{i}$ over $T_{i}$ such that $V \cong V_{i} \times_{T_{i}} T$, moreover by taking a greater index if necessary the morphism $V_{i} \rightarrow T_{i}$ is finite \cite[8.5.5]{EGA4iii}. $V_{i}$ is not necessarily a torsor over $T_{i}$, but this construction is functorial, up to the choice of a bigger index, over $V$ and commutes with fibered products by \cite[8.8.2 (i)]{EGA4iii} (see \cite[\S 8.8.3]{EGA4iii} for more details). This means that if we take the projective system $\{G_{i}\}_{i \in I}$ and apply \cite[8.8.2.5]{EGA4iii} over the fact that there is an isomorphism $V \times_{T} G_{T} \cong V \times_{T} V$ means that there exists a big enough index $i$ such that we have an isomorphism $V_{i} \times_{T_{i}} G_{i} \cong V_{i} \times_{T_{i}} V_{i}$, making $V_{i} \rightarrow T_{i}$ a $G$-torsor. It is easy to see that $V_{i}$ is Nori-reduced if $V$ is. \\
For part (2), we first chose two indexes $i$ and $j$ such that $V$ descends to $V_{i}$ over $T_{i}$ and $W$ so does to $W_{j}$ over $T_{j}$. We can assume $i \leq j$ as we can take another bigger index to both $i$ and $j$, in that case we can take $V_{j}=V_{i} \times_{T_{i}} T_{j}$. Both schemes satisfy the hypotheses of Th\'eor\`eme 8.8.2 (i) of \cite{EGA4iii} and thus we have a bijection between $T$-morphisms between $V$ and $W$ and the directed limit of the sets $\text{Hom}_{T_{l}}(V_{l},W_{l})$ where $l \geq j$ and $V_{l}=V_{k} \times_{T_{j}} T_{l}$ and $W_{l}=W_{k} \times_{T_{j}} T_{l}$. This means that $\phi$ can be seen as a directed system of morphisms $\phi_{l}:V_{l} \rightarrow W_{l}$ and by picking a possibly bigger index we can assume that it is compatible with the actions of the respective group-schemes, as we can apply loc. cit. to get a bijection for $\text{Hom}_{T}(G_{T},H_{T})$ and $\text{Hom}_{T}(G_{T} \times_{T} V, H_{T} \times_{T} W)$ and its corresponding directed system of $\text{Hom}$-sets for each respective Hom-set. This bijection allows us to conclude that the commutative diagram
$$ \xymatrix{G_{T} \times V \ar[d] \ar[r] & H_{T} \times W \ar[d] \\
	V \ar[r]  & W}  $$ 
that compatibilizes the actions descends to a diagram that makes the actions compatible on a certain index, bigger that $j$, which means that we have a morphism of torsors over $T_{j}$. Finally, the isomorphism part of the second statement comes from \cite[8.8.2.5]{EGA4iii}.
\endproof
We remark that the results of \cite{EGA4iii} for projective limits of schemes, imply that a pro-NR torsors do not have non-trivial sub-torsors, making them Nori-reduced (Defintion \ref{definition:Nori-reduced_torsor}). \\
The main result of this section is the following, now $k$ will be a perfect field for the rest of this chapter:
\begin{proposition} \label{proposition:Pro-saturated_torsors_possess_a_FGS}
	Let $X$ be a reduced and proper scheme of finite type over a field $k$ with a rational point $x \in X(k)$. Let $\{T_{i}\}_{i \in I}$ be a projective system of pointed Nori-reduced torsors and $T$ its projective limit with $t \in T(k)$ over $x$. \\
	Then $T$ possesses a FGS, as it has the same universal torsor as $X$, and we have $\pi_{1}^{N}(T,t)=\ker\{\pi_{1}^{N}(X,x) \rightarrow \displaystyle{\lim_{\leftarrow}} G_{i} \}$ where $G_{i}$ is the finite group-scheme associated with to $T_{i}$.
\end{proposition}
To show this, we need an extension of Proposition \ref{prop:Closure_of_towers_for_schemes_having_a_FGS}
\begin{lemma} \label{lemma:Closure_of_towers_of_torsors_for_pro-NR_torsors}
	Let $X$ be a proper, reduced and connected scheme of finite type over a field $k$ with a rational point $x \in X(k)$ and let $T$ be a pro-NR torsor with $t \in T(k)$ over $x$.
	Then, if $W \rightarrow V \rightarrow T$ is a tower of finite pointed torsors, there exists a Nori-reduced closure $U \rightarrow T$ of the tower that mimics the properties of the closure of a tower of finite torsors over $X$, outlined in Definition \ref{definition:Envelopes_and_closure_of_a_tower_of_torsors}.
\end{lemma}
\proof
Let us start with the existence of an envelope for the tower: Let $W \rightarrow V \rightarrow T$ be a tower of torsors. Applying Lemma \ref{lemma:Torsors_over_pro-NR_torsors_descent_to_a_finite_torsor} to $V$ to obtain a cartesian square
$$ \xymatrix{ V \ar[d] \ar[r] & V_{i} \ar[d] \\
	T \ar[r] & T_{i} } $$
where $i \in I$ and $V_{i} \rightarrow T_{i}$ is a torsor. Applying \cite[8.8.2 (ii)]{EGA4iii} to the composition $W \rightarrow V \rightarrow T$ we also get the following diagram in which all possible squares are cartesian:
$$\xymatrix{W \ar[r] \ar[d] & W_{i} \ar[d]  \\
	V \ar[d] \ar[r] & V_{i} \ar[d] \\
	T \ar[r] & T_{i} }$$
but it is not necessarily clear if $W_{i} \rightarrow V_{i}$ is a torsor. This is indeed the case (over a possibly larger index) as we can consider the projective system $\{V_{j} \}_{j \geq i}$ with base $V_{i}$ and projective limit $V$ (see \cite[8.2.5]{EGA4iii}) and apply Lemma \ref{lemma:Torsors_over_pro-NR_torsors_descent_to_a_finite_torsor} for $W \rightarrow V$. \\
Now we can utilize Proposition \ref{prop:Closure_of_towers_for_schemes_having_a_FGS} over the tower of finite pointed torsors $W_{i} \rightarrow V_{i} \rightarrow T_{i}$ as all schemes are of finite type and $T_{i}$ possesses a FGS (Proposition \ref{prop:FGS_of_Nori-reduced_torsors}), and thus we obtain a closure $U_{i}$ of this tower, giving us the commutative diagram
$$ \xymatrix{W \ar[r] \ar[d] & W_{i} \ar[d] & \\
	V \ar[d] \ar[r] & V_{i} \ar[d] &  U_{i} \ar[lu] \ar[l] \ar[ld] \\
	T \ar[r] & T_{i}& } $$
and if we denote by $U$ the pull-back of $U_{i}$ to $T$, we obtain  the following diagram
$$ \xymatrix{& W \ar[r] \ar[d] & W_{i} \ar[d] & \\
	U \ar[ru] \ar[r] \ar[rd] &V \ar[d] \ar[r] & V_{i} \ar[d] &  U_{i} \ar[lu] \ar[l] \ar[ld] \\
	&T \ar[r] & T_{i}& } $$
where we see that the left-hand side of the diagram shows the existence of a torsor that envelopes the tower $W \rightarrow V \rightarrow T$ as $U_{i}$ satisfies the properties that we require for an envelope of a tower of torsors, but $U$ might not be a minimal torsor that envelops the tower. This process shows how to obtain an envelope for the tower $W \rightarrow V \rightarrow T$ from the closure of a tower of torsors over the finite torsor $T_{i}\rightarrow X$. \\
Now let us assume that both $W \rightarrow V$ and $V \rightarrow T$ are Nori-reduced and let us show the existence of a closure in this case. First, we notice that as $U_{i}$ is the closure of a tower of Nori-reduced torsor, we have that $U_{i} \rightarrow T_{i}$ is Nori-reduced but its pull-back $U \rightarrow T$ might not be. At least, both arrows $U \rightarrow V$ and $U \rightarrow W$ are torsors and then, if we take the maximal Nori-reduced sub-torsor $\bar{U} \hookrightarrow U$, we obtain a Nori-reduced envelope that additionally is a torsor over $W$, so we can suppose from now on that $U$ satisfies these properties. \\
The existence of a unique Nori-reduced closure in the case both torsors in the tower $W \rightarrow V \rightarrow T$ are Nori-reduced comes from applying Zorn's lemma to the (non-empty) set of isomorphism classes of Nori-reduced envelopes of this tower, i.e., we are considering the skeletal sub-category of the category of envelopes with morphisms of torsors as arrows. We will abuse notation when considering classes and individual torsors. \\
We consider over the classes, the partial ordering $U \leq U^{\prime}$ iff there exists a morphism of torsors $U \rightarrow U^{\prime}$ of over $T$, which holds as we have the following properties:
\begin{itemize}
	\item[i)] Every chain of Nori-reduced envelopes has at most a finite amount of members.
	\item[ii)] The poset of Nori-reduced envelopes is directed. 
\end{itemize}
For i), if $U \leq U^{\prime}$, let $G$ be the group-scheme associated to $U \rightarrow T$. Then, as both torsors are Nori-reduced over $T$, $U^{\prime}$ is of the form $U/N$ where $N$ is a normal subgroup of $G$. As there is only a finite amount of quotients of $G$, there is a finite amount of envelopes greater than $U$. \\ 
To prove ii), let $U,U^{\prime}$ be two Nori-reduced envelopes of the tower $W \rightarrow V \rightarrow T$ and let $l \in I$ be an index such that both torsors $U$ and $U^{\prime}$ descend to torsors $U_{l}$ and $U_{l}^{\prime}$ respectively, this is possible as $I$ is directed. In that case, these torsors envelop the tower $W_{l} \rightarrow V_{l} \rightarrow T_{l}$ that descends from the tower over $T$ and thus if $Z_{l}$ is the closure of this tower, we have morphisms of torsors $U_{k},U_{k}^{\prime} \rightarrow Z$ which are quotients in this case. This implies that the pull-back $Z$ of $Z_{k}$ over $T$ is a Nori-reduced common quotient of $U$ and $U^{\prime}$ that envelops the tower $W \rightarrow V \rightarrow T$, showing that the poset of Nori-reduced envelopes of this tower is directed. \\
For the general case, we first note that as long as $V \rightarrow T$ is Nori-reduced, we will have a closure for the tower even if $W \rightarrow V$ is not Nori-reduced from the last paragraph, but of course, this closure might not longer be a torsor over $W$. This ties into the general case, because if $V \rightarrow T$ is no longer Nori-reduced, any Nori-reduced envelope $U$ of the tower goes faithfully flat over $\bar{V} \subset V$, the canonical Nori-reduced sub-torsor of $V$, and then, if $U$ envelopes $W \rightarrow V \rightarrow T$ then it will envelope the tower $W \times_{V} \bar{V} \rightarrow \bar{V} \rightarrow T$, such that we have a commutative diagram
$$ \xymatrix{&W \times_{V} \bar{V} \ar[d] \ar[r] & W \ar[d] \\
	U \ar[ur] \ar[r] \ar[drr] &\bar{V} \ar[r] \ar[dr] & V \ar[d] \\
	& & T} $$
and then we conclude that we have a closure for the initial given tower in this case as well, finishing the proof.
\endproof
\proof[Proof of Proposition 2.3]
Let $\hat{X}$ be the universal torsor of $X$ and let $Y$ be a finite Nori-reduced pointed torsor over $T$, applying Lemma \ref{lemma:Torsors_over_pro-NR_torsors_descent_to_a_finite_torsor}, there exist an index $i$ and a pointed torsor $Y_{i} \rightarrow T_{i}$ making a cartesian diagram
$$ \xymatrix{Y \ar[d] \ar[r] & Y_{i} \ar[d] \\
	T \ar[r] & T_{i} }. $$
As $T_{i}$ is Nori-reduced, from Proposition \ref{prop:FGS_of_Nori-reduced_torsors}, $\hat{X}$ is also the universal torsor of $T_{i}$ and thus we have an arrow $\hat{X} \rightarrow Y_{i}$. Combining this arrow with the canonical arrow $\hat{X} \rightarrow T$ we obtain an arrow $\hat{X} \rightarrow Y \cong Y_{i} \times_{T_{i}} T$ and thus we conclude that all Nori-reduced torsors over $T$ are a quotient of $\hat{X}$, this means that $T$ possesses a universal pointed torsor $\hat{T}$ and we have a natural morphism of torsors over $T$, $\hat{X} \rightarrow \hat{T}$. \\
To get an isomorphism, it suffices to prove that $\hat{T}$ has a trivial fundamental group-scheme: Let $Z \rightarrow \hat{T}$ be a finite Nori-reduced torsor over $\hat{T}$, applying Lemma \ref{lemma:Torsors_over_pro-NR_torsors_descent_to_a_finite_torsor} over $T$, which is quasi-compact and quasi-separated over $k$ because $T$ is affine over $k$, there exist a finite Nori-reduced torsor $V \rightarrow T$ and a finite Nori-reduced torsor $W \rightarrow V$ fitting into a commutative diagram where all possible squares are cartesian:
$$\xymatrix{Z \ar[r] \ar[d] & W \ar[d]  \\
	\hat{T} \ar[dr] \ar[r] & V \ar[d] \\
	& T }.$$
From Lemma \ref{lemma:Closure_of_towers_of_torsors_for_pro-NR_torsors}, there exists a Nori-reduced closure $U \rightarrow T$ of the tower $W \rightarrow V \rightarrow T$, and thus we have a canonical arrow $\hat{T} \rightarrow U$ which composed with the arrow $U \rightarrow W$ gives a section $\hat{T} \rightarrow Z=W \times_{V} \hat{T}$, making $Z$ a trivial torsor and finishing the proof. 
\endproof
A particular case of the fact that the universal torsor $\hat{T}$ does not possess non-trivial finite Nori-reduced torsors coming from the proof Proposition \ref{proposition:Pro-saturated_torsors_possess_a_FGS} is the following:
\begin{corollary} \label{corollary:The_universal_torsor_has_a_trivial_fundamental_gruop}
	Let $X$ be a proper, reduced and connected scheme of finite type over a field $k$ with a rational point $x \in X(k)$. Then, if $\hat{X}$ is the universal torsor of $X$, all finite Nori-reduced torsors over $\hat{X}$ are trivial.
\end{corollary}
\proof
If $Z \rightarrow \hat{X}$ is a finite Nori-reduced torsor, we can apply Lemma \ref{lemma:Torsors_over_pro-NR_torsors_descent_to_a_finite_torsor} to a commutative diagram
$$ \xymatrix{Z \ar[r] \ar[d] & W \ar[d]  \\
	\hat{X} \ar[dr] \ar[r] & V \ar[d] \\
	& X } $$
where $W \rightarrow V \rightarrow X$ is a tower of torsors. From here, using closure of towers directly we can argue analogously as we did in the proof of Proposition \ref{proposition:Pro-saturated_torsors_possess_a_FGS} to obtain a section $\bar{X} \rightarrow Z$.
\endproof
From Proposition \ref{proposition:Pro-saturated_torsors_possess_a_FGS}, we get the following short exact sequence for a pro-NR torsor
\begin{equation} \label{equation:Short_exact_sequence_pro-NR_torsors}
1 \rightarrow \pi_{1}^{N}(T,t) \rightarrow \pi_{1}^{N}(X,x) \rightarrow \lim_{\leftarrow} G_{i} \rightarrow 1
\end{equation}
for suitable rational points. Also, from Corollary \ref{corollary:The_universal_torsor_has_a_trivial_fundamental_gruop}, the universal torsor $\hat{X}$ of $X$ has a trivial FGS.
\section{Pull-back of torsors to the geometric generic fiber} \label{sec:Pull-back_to_the_geometric_generic_fiber}
We start by introducing some terminology:
\begin{definition} \label{definition:Pure_pull-back_and_mixed_torsor}
	Let $T$ be a $G$-torsor over $X$. $T$ is a \textbf{pull-back torsor} if there exists a torsor $S^{\prime} \rightarrow S$ such that its pull-back along $f$, $S^{\prime} \times_{S} X \rightarrow X$ is $T$, trivial torsors are considered pull-back torsors.\\ 
	We say that a non-trivial torsor $T$ is \textbf{pure} with respect to $S$ if $T$ and none of its non-trivial quotients is the pull-back of a torsor over $S$. \\
	A torsor which is neither a pull-back nor a pure torsor is called a \textbf{mixed torsor}.
\end{definition}
The purpose of this chapter, is to prove that for a morphism of proper reduced and connected $k$-schemes $f:X \rightarrow S$ with $S$ integral and $k$ algebraically closed, the pull-back $T_{\bar{\eta}} \rightarrow X_{\bar{\eta}}$ of a pure and Nori-reduced $G$-torsor $T \rightarrow X$ is Nori-reduced over the geometric generic fiber, where $\eta$ is the generic point of $S$. For this, we need two technical results involving finite group-schemes that we will outline in the next two subsections.
\subsection{Some group-scheme results} \label{subsec:Group-scheme_results}
We recall that for a scheme $X$ of finite type over $k$, $\tilde{X}:\text{Alg}_{k}^{0} \rightarrow \text{Set}$ will denote its functor of points.
\subsubsection{Well-corresponded subgroup-schemes} \label{subsubsec:Well_corresponded_subgroup-schemes}
This section is mostly independent from the rest, we will consider an extension $k \subset L$ of algebraically closed fields. For a finite group-scheme $G$, we will denote $G_{L}$ its base change to $L$. All algebras considered will be finite over their respective base field.
\begin{defi} \label{defi:Good_correspondence_of_subgroups}
	Let $G$ be a finite $k$-group-scheme and $G_{L}$ its base change to $L$. We say that a subgroup-scheme $H \subset G_{L}$ is \textbf{well-corresponded} if there exists a subgroup-scheme $H^{\prime} \subset G$ such that $H_{L}^{\prime} \cong H$.
\end{defi}
\begin{remark} \label{remark:There_is_at_most_one_subgroup_scheme_which_is_corresponded}
	If $H$ is well-corresponded, there is a unique subgroup-scheme of up to isomorphisms of $G$ whose base change to $L$ is $H$.
\end{remark}
To prove that all subgroup-schemes of a finite group-scheme $G$ are well-corresponded, we will use the local-\'etale exact sequence
$$ 1 \rightarrow G^{0} \rightarrow G \rightarrow G^{\text{\'et}} \rightarrow 1 $$
which is split in our case as $k$ is perfect, this means that we actually have an isomorphism $G \cong G^{0} \rtimes G^{\text{\'et}}$ coming from the retraction $G_{\text{red}} \cong G^{\text{\'et}}$. The exact sequence and the splitting commute with extensions of the base field (see \cite[Prop. 2.37]{MilneAG2017}), in particular $\left(G^{0}\right)_{L}=G_{L}^{0}$ and $\left(G^{\text{\'et}}\right)_{L}=G_{L}^{\text{\'et}}$ with $G_{L} \cong G_{L}^{0} \rtimes G_{L}^{\text{\'et}} $. This exact sequence is unique in the sense that, if we have another short exact sequence
$$1 \rightarrow N \rightarrow G \rightarrow G/N \rightarrow 1$$
with $N$ local and $G/N$ \'etale, we have isomorphisms $N \cong G^{0}$ and $G/N \cong G^{\text{\'et}}$ that carry one exact sequence to the other. \\
As \'etale and local group-schemes are the building blocks of the exact sequence, we will prove well-correspondence for subgroup-schemes of \'etale and local group-schemes separately.
\begin{proposition} \label{prop:Well_correspondence_for_etale_and_local_finite_group_schemes}
	Let $G$ be a finite group-scheme.
	\begin{itemize}
		\item If $G$ is \'etale, all subgroup-schemes of $G_{L}$ are well-corresponded.
		\item If $G$ is local, all subgroup-schemes of $G_{L}$ are well-corresponded.
	\end{itemize}
\end{proposition}
\proof
We start with the \'etale case, if $G$ is \'etale, as $k$ and $L$ are algebraically closed, $G$ is a constant group-scheme associated to an abstract group $\Gamma$. In particular $G=\Spec(k^{\Gamma})$ where $k^{\Gamma}$ is the Hopf $k$-algebra spanned by the elements $(e_{g})_{g \in \Gamma}$ where $e_{g}:\Gamma  \rightarrow k$ is the map defined as $e_{g}(h)=\delta_{gh}$ with coalgebra operations
$$\begin{array}{rcl}
\Delta(e_{g}) & = & \sum_{hj = g} e_{h} \otimes e_{j} \\
\epsilon(e_{g})& = &\delta_{eg} \\
S(e_{g}) &= & e_{g^{-1}} 
\end{array}$$
where $\delta$ is the Kronecker delta, $e \in \Gamma$ is the unit element and $\Delta$, $\epsilon$ and $S$ are the comultiplication, counit and antipode morphisms respectively. We can easily see that $G_{L}$ has the same Hopf algebra except that we use $L$ instead of $k$. From that fact we see that the coalgebra formulas remain unchanged after base change to $L$. As subgroup-schemes of $G$ in this case are constant group-schemes associated with subgroups of $\Gamma$, we conclude the proposition in this case. \\
For the local case, as $k$ and $L$ are perfect, from the structure theorem for finite local group-schemes \cite[Theorem 11.29]{MilneAG2017}, if we write $G=\Spec(A)$, we have that
\begin{equation} \label{eq:Hopf_algebra_of_a_local_group_scheme_over_a_perfect_field}
A \cong k[x_{1},x_{2},\cdots,x_{n}]/(x_{1}^{p^{r_{1}}},\cdots,x_{n}^{p^{r_{n}}})
\end{equation}
for some integers $r_{i} \geq 1$. 
If we take the base change $A_{L} = A \times_{k} L$, as a Hopf algebra, it has formulas for its coalgebra structure with coefficients over $k$. \\
Now let $H_{L}$ be a subgroup-scheme of $G_{L}$, we see immediately that $H_{L}$ is connected because $G_{L}$ is topologically a point. In that case, if $H_{L} = \Spec(B_{L})$, we know that $B_{L}$ is a quotient of $A_{L}$, which has the same coalgebra formulas as $A_{L}$, with coefficients over $k$, but it may have a different presentation with polynomials with coefficients over $L \backslash k$. \\
We claim that in fact if $B$ is the Hopf algebra corresponding to a subgroup-scheme $H$ of $G=\Spec(A)$, we have
$$ B \cong k[x_{1},x_{2},\cdots,x_{n}]/(x_{1}^{p^{s_{1}}},\cdots,x_{n}^{p^{s_{n}}}) $$
with $0 \leq s_{i} \leq r_{i}$ for all $1 \leq i \leq n$ which implies the proposition in the local case. \\
For this, we first recall that if $F^{r}:G \rightarrow G^{(p^{r})}$ is the $r$-th iteration of the relative Frobenius morphism, its image corresponds to the inclusion $A^{p^{r}} \rightarrow A$, this is a sub-Hopf algebra of $A$. We say that $G$ has \textbf{height} $\leq r$ for $r \geq 1$ if $F^{r}$ has a trivial image, and we say that $G$ has \textbf{height} $r$ if it is of height $\leq r$ but not of height $\leq r-1$. Finite local group-schemes over $k$ are precisely the algebraic group-schemes of finite height. Let $I$ be the augmentation ideal of $A$, i.e. the kernel of the counit map $\epsilon:A \rightarrow k$, if $G$ has height $\leq r$, then $x^{p^{r}}=0$ for all $x \in I$. \\
We will proceed with the proof of the claim by induction on the height of $G$. If $G$ has height one, we have a special case of equation \ref{eq:Hopf_algebra_of_a_local_group_scheme_over_a_perfect_field}:
$$ A \cong  k[x_{1},x_{2},\cdots,x_{n}]/(x_{1}^{p},\cdots,x_{n}^{p}) $$
where the elements $x_{i} \in I$ ($1 \leq i \leq n$) map to a base of the $k$-vector space $I/I^{2}$ \cite[Proposition 11.28]{MilneAG2017}. If $B$ is a Hopf algebra quotient of $A$, and $I^{\prime}$ is its augmentation ideal, we have a surjection $I/I^{2} \rightarrow I^{\prime}/\left(I^{\prime}\right)^{2}$. Then, the claim holds for $H=\Spec(B)$ as it also has height one and we can identify those $x_{i}$ that do not map to zero with their classes in $I^{\prime}$ using the quotient morphism $A \rightarrow B$. \\
For the general case, from the proof of \cite[Theorem 11.29]{MilneAG2017}, we can be more specific with the presentation of $A$: we can write it as
\begin{equation} \label{eq:Hopf_algebra_of_a_local_group_scheme_over_a_perfect_field_more_detailed}
A \cong k[y_{1},y_{2},\cdots,y_{m},z_{1},\cdots,z_{k}]/(y_{1}^{p^{u_{1}+1}},\cdots,y_{m}^{p^{u_{m}+1}},z_{1}^{p},\cdots z_{k}^{p})
\end{equation}
where the elements $y_{i} \in I$ ($1 \leq i \leq m$) come from the presentation of $A^{p}$
$$A^{p} \cong k[t_{1},t_{2},\cdots,t_{m}]/(t_{1}^{p^{u_{1}}},\cdots,t_{m}^{p^{u_{m}}})$$
where $y_{i}^{p} = t_{i}$, and the elements $z_{j} \in I$ ($1 \leq j \leq k$) are such that their classes $\bar{z_{i}}$ in the $k$-vector space $I/I^{2}$ form a maximal subset such that each $z_{i}^{p}=0$ and the set $\{\bar{z_{1}}, \cdots , \bar{z_{k}}\}$ is linearly independent. One key part in the proof is that the set $\{\bar{y_{i}},\bar{z_{j}}\}_{i,j}$ is a basis for $I/I^{2}$, in particular, the complement of the vector subspace spanned by the $\bar{y_{i}}$ is the subspace spanned by the $\bar{z_{j}}$. \\
Now let $B$ be a Hopf algebra which is a quotient of $A$, and let $J$ be the Hopf ideal of $A$ such that $B=A/J$. If $I^{\prime}$ is the augmentation ideal of $B$, we have then a surjective morphism $T:I/I^{2} \rightarrow I^{\prime}/\left(I^{\prime}\right)^{2}$. As the quotient morphism $A \rightarrow B$ maps $A^{p}$ onto $B^{p}$ and $A^{p}$ has one fewer height than that of $A$, using the induction hypothesis, we see that the claim holds for $B^{p}$ and thus we have
$$ B^{p} \cong k[t_{1},t_{2},\cdots,t_{m}]/(t_{1}^{p^{v_{1}}},\cdots,t_{m}^{p^{v_{m}}}) $$
with $0 \leq v_{i} \leq u_{i}$. If we discard the elements with $v_{i}=0$ we see that $T$ maps $\left\langle \bar{y_{1}},\cdots, \bar{y_{m}} \right\rangle$ to the subspace $V$ of $I^{\prime}$ spanned by the images of the elements $y_{i}$ with $v_{i} \neq 0$. On the other hand, if $v_{i}=0$ this means that $y_{i}^{p} \in J$ and then these elements, together with the $z_{j}$, map to elements of $I^{\prime}$ with zero $p$-th power, which is the complement of $V$. \\
Finally, if we suppose that for all $1 \leq i \leq l$ we have $v_{i} \neq 0$ while $v_{i}=0$ for $l < i \leq m$, as long a none of the elements forming the base of $I/I^{2}$ belong to $J$, we will have that
$$ B \cong k[y_{1},y_{2},\cdots,y_{m},z_{1},\cdots,z_{k}]/(y_{1}^{p^{v_{1}+1}},\cdots,y_{l}^{p^{v_{l}+1}},y_{l+1}^{p},\cdots,y_{m}^{p} ,z_{1}^{p},\cdots z_{k}^{p}).  $$
Otherwise, the corresponding exponent will be 0 for each element of the base belonging to $J$ finishing the proof. 
\endproof
Before the proof of the general result for finite group-schemes, we introduce the group-schematic version of the normalizer:
\begin{definition} \label{definition:Normalizing_subgroup_and_normalizer}
	Let $G$ be an algebraic group-scheme over $k$ and let $H,N$ be subgroup-schemes of $G$. We say that $N$ \textbf{normalizes} $H$ if for any $k$-algebra $R$ the group $\widetilde{N}(R)$ normalizes $\widetilde{H}(R)$, this just means that $\widetilde{N}(R)\widetilde{H}(R)\widetilde{N}(R)^{-1}=\widetilde{H}(R)$ for all $R$. \\
	The \textbf{normalizer} of $H$ in $G$ is the subgroup-scheme $N_{G}(H)$ of $G$ corresponding to the functor
	$$ N_{G}(R) = \{ g \in \widetilde{G}(R): \, g_{S}\widetilde{H}(S)g_{S}^{-1} =\widetilde{H}(S) \text{ for all $R$-algebras $S \in \text{Alg}_{k}^{0}$} \} $$ 
	where for a $R$-algebra $R \rightarrow S$, $g_{S}$ denotes the image of $g$ by the induced morphism.
\end{definition} 
\begin{remark} \label{remark:Remarks_normalizers}
	The functor defining the normalizer is representable \cite[Prop. 1.83]{MilneAG2017}, hence $N_{G}(N)$ is an actual subgroup-scheme of $G$. It is easy to see that $N_{G}(H)$ normalizes $H$. \\
	In fact, directly from the definition, we conclude that a subgroup-scheme $N$ of $G$ normalizes $H$ if and only if $N \subset N_{G}(H)$ and hence the normalizer is the biggest subgroup-scheme of $G$ that normalizes $H$. In particular, a subgroup-scheme $H$ is normal in $G$ if and only if $G=N_{G}(H)$. \\
	Finally, it is important to remark that the formation of normalizers commutes with extension of the base field, this means that for an extension $k \subset L$, $\left( N_{G}(H) \right)_{L} =N_{G_{L}}(H_{L})$.
\end{remark}
\begin{proposition} \label{prop:Good_correspondence_for_finite_group_schemes}
	Let $G$ be a finite group-scheme over the algebraically closed field $k$ and let $L$ be an algebraically closed extension of $k$. Then any subgroup-scheme of $G_{L}$ is well-corresponded.
\end{proposition}
\proof 
First we recall that in the local-\'etale exact sequence
$$ 1 \rightarrow G^{0} \rightarrow G \rightarrow G^{\text{\'et}} \rightarrow 1 $$
the group-schemes $G^{0}$ and $G^{\text{\'et}}$ are universal in the sense that any morphism $G^{\prime} \rightarrow G$ with $G^{\prime}$ local factors through $G^{0}$, and analogously, any morphism $G \rightarrow G^{\prime}$ with $G^{\prime}$ \'etale factors through $G^{\text{\'et}}$. \\
From this we can easily show that for a subgroup-scheme $G^{\prime} \subset G$ we have the following diagram
$$ \xymatrix{1 \ar[r] &\left(G^{\prime} \right)^{0} \ar[r]  \ar[d] &G^{\prime} \ar[r] \ar[d] & \left(G^{\prime} \right)^{\text{\'et}}  \ar[r] \ar[d] & 1\\
	1 \ar[r] & G^{0} \ar[r] & G \ar[r] & G^{\text{\'et}} \ar[r] & 1} $$
with exact rows, commutative squares, and each vertical arrow is an inclusion of subgroup-schemes. \\
Now we consider a subgroup-scheme $H_{L}$ of $G_{L}$, from the fact that the local-\'etale connected sequence commutes with extension of base fields and from Proposition \ref{prop:Well_correspondence_for_etale_and_local_finite_group_schemes}, we see that there exist two subgroup-schemes $H^{\prime} \subset G^{0}$ and $H^{\prime \prime} \subset G^{\text{\'et}}$ such that $\left( H^{\prime} \right)_{L} = \left(H_{L}\right)^{0}$ and $ \left( H^{\prime \prime} \right)_{L} = \left(H_{L}\right)^{\text{\'et}} $. As $L$ is perfect, we also have that $H_{L} \cong \left(H_{L}\right)^{0} \rtimes M$ where $M \subset \left( G_{L} \right)_{\text{red}}$ is the restriction to $H^{\text{\'et}}$ of the retract $i_{L}:\left( G_{L} \right)_{\text{red}} \overset{\sim}{\rightarrow} G_{L}^{\text{\'et}}$ which is the same as $H_{\text{red}}$. \\
As $\left( G_{L} \right)_{\text{red}}$ is \'etale, we have an \'etale subgroup-scheme $J \subset G_{\text{red}}$ with $J_{L}=M$ as the latter subgroup-scheme is well corresponded, because the isomorphism $i_{L}$ descends to the retraction $i:G_{\text{red}} \overset{\sim}{\rightarrow} G^{\text{\'et}}$, $J$ is isomorphic to $H^{\prime \prime}$ via this retraction. \\
We claim that we can form the subgroup-scheme $H^{\prime} \rtimes J$ and its base change to $L$ is $H_{L}$. The base change part comes from the fact that $H_{L} = H_{L}^{0} \rtimes M$ and thus we only need to prove that $H^{\prime} \cap J$ is trivial, and that $J$ normalizes $H^{\prime}$. The intersection is trivial because $J \cap H^{\prime} \subset G_{\text{red}} \cap G^{0} = 1$ and $J$ normalizes $H^{\prime}$ because as $M$ normalizes $H_{L}^{0}$, we have an inclusion $M \subset N_{G_{L}}(H^{0})$. But the formation of a normalizer commutes with extensions of the base field, and thus $\left(N_{G}(H^{\prime}) \right)_{L} =N_{G_{L}}(H^{0})$ which implies that $J \subset N_{G}(H^{\prime})$, finishing the proof.
\endproof
Now we will apply this result to the following: Let $f:X \rightarrow S$ be the a morphism of proper reduced and connected $k$-schemes, with $k$ algebraically closed, we consider $L=\overline{\kappa(\eta)}$ where $\eta$ is the generic point of $S$, and we will also consider $\bar{\eta}$ the geometric generic point of $S$. Then for a pointed $G$-torsor $T \rightarrow X$ we can consider its pull-back to the geometric generic fiber $T_{\bar{\eta}} \rightarrow X_{\bar{\eta}}$ which is a pointed $G_{L}$-torsor. We then have morphisms over $k$ and $L$ resp. $\pi_{1}^{N}(X) \rightarrow G$ and $\pi_{1}^{N}(X_{\bar{\eta}}) \rightarrow G_{\bar{\eta}}$ and as a consequence of  Proposition \ref{prop:Good_correspondence_for_finite_group_schemes}, we have the following:
\begin{corollary} \label{corollary:Good_correspondence_of_subgroups_implies_good_correspondence_of_quotient_torsors}
	Let $f:X \rightarrow S$ be the a morphism of proper reduced and connected $k$-schemes, with $k$ algebraically closed. Let $T \rightarrow X$ be a pointed $G$-torsor and let $T_{\bar{\eta}} \rightarrow X_{\bar{\eta}}$ be its pull-back to the geometric generic fiber. Then, there exists a one-to-one correspondence between quotients of $T$ and quotients of $T_{\bar{\eta}}$
\end{corollary}
\proof
Let $L$ be the residue field of the geometric point $\bar{\eta}$. Quotients of $T$ (resp. $T_{\bar{\eta}}$) are in one-to-one correspondence with quotients of $G$ (resp. $G_{L}$) and these are in one-to-one correspondence with subgroup-schemes of $G$ (resp. $G_{L}$). We know that subgroup-schemes of $G_{L}$ are well-corresponded by Proposition \ref{prop:Good_correspondence_for_finite_group_schemes}, and additionally, if $H$ is a normal subgroup-scheme of $G$, the fact that $H$ is normal if and only if $N_{G}(H)=G$ and the formation of normalizer commutes with extensions of the base field implies that the descent of a normal subgroup-scheme of $G_{L}$ is normal in $G$. \\
If $H$ is not normal, both quotients $G/H$ and $G_{L}/H_{L}$ exists and likewise for the quotient schemes $T/H$ and $T_{\bar{\eta}}/H_{L}$ as $H$ is well-corresponded with $H_{L}$. By looking for example, at the underlying algebras defining the group-scheme quotients and the associated essentially finite bundles/covers (see Definition \ref{defi:Weak_quotients} and Remark \ref{remark:About_weak_quotients} below), we can conclude the result in this case as well.
\endproof
From these results, we will abuse notation on the rest of this article and omit the subscript alluding to the field for finite group-schemes unless it is needed for clarity.
\subsubsection{Core of a subgroup-scheme} \label{subsubsec:Core_of_a_subgroup-scheme}
Let $k$ be a perfect field. We start with a definition:
\begin{defi} \label{defi:Weak_quotients}
	Let $t:T \rightarrow X$ be a $G$-torsor with $G$ finite, if $H \subset G$ is a non-normal subgroup-scheme of $G$, we will call the quotients by the restricted multiplication action \textbf{weak quotients}. It is a scheme $G/H$ with a faithfully flat arrow $G \rightarrow G/H$. \\ 
	In the case of a torsor $t:T \rightarrow X$, we will call the quotient by the restricted $H$-action a \textbf{weak quotient} when $H$ is no normal. It is a scheme $T/H$ with a finite and faithfully flat morphism $t^{\prime}:T/H \rightarrow X$, we have also a faithfully flat morphism $q:T \rightarrow T/H$.
\end{defi}
\begin{remark} \label{remark:About_weak_quotients}
	Both arrows $G \rightarrow G/H$ and $T \rightarrow T/H$ are $H$-torsors. In the case of group-schemes the quotient $G/H$ is a well-defined scheme, its functor of points $\widetilde{G/H}$ is the fppf (actually fpqc) sheafification of the pre-sheaf $R \mapsto \tilde{G}(R)/\tilde{H}(R)$ over the category of affine $k$-schemes of finite type with the fppf topology, for more details see \cite{RaynaudQuotient} and \cite[\S 5 c.]{MilneAG2017}, the second reference is valid for any quotient of group-schemes of finite type. The quotient of group-schemes possesses a natural left $G$-action $G \times_{k} G/H \rightarrow G/H$ that comes from the left action of $\tilde{G}$ over $\tilde{G}/\tilde{H}$ by multiplication on the left. \\
	In the particular case of finite group-schemes, if $G=\Spec(A)$ where $A$ is a Hopf-algebra over $k$, then $G/H = \Spec(A^{H})$ as stated in Section \ref{subsec:Global_sections_ess_fin_bundles}. The morphism $t^{\prime}:T/H \rightarrow X$ is an essentially finite cover such that we have an inclusion of bundles $\left( t^{\prime} \right)_{*}(\Strsh{T/H}) \subset t_{*}(\Strsh{T})$ which corresponds to the inclusion of algebras $A^{H} \hookrightarrow A$ as elements of $\text{Rep}_{G}$ under tannakian correspondence.
\end{remark}
Now we will introduce a group-theoretic notion:
\begin{defi} \label{defi:Core_of_an_abstract_subgroup}
	Let $\Gamma$ be an abstract group and let $J \subset \Gamma$ be a subgroup. The \textbf{core of $J$} is the subgroup of $\Gamma$ defined as
	$$ \text{Core}_{\Gamma}(J)=\bigcap_{g \in \Gamma} gJg^{-1}. $$
\end{defi}
We can further characterize this group.
\begin{proposition} \label{prop:Properties_of_the_core_of_an_abstract_subgroup}
	Let $\Gamma$ be an abstract group and let $J \subset \Gamma$ be a subgroup. Then, $\text{Core}_{\Gamma}(J)$ is the biggest normal subgroup of $\Gamma$ contained in $J$, and if we consider $n=[\Gamma:J]$ the index of $J$, the core is also the kernel of the morphism $\Gamma \rightarrow S_{n}$ associated to the action $\mu:\Gamma \times \Gamma/J \rightarrow \Gamma/J$ of left multiplication by $\Gamma$ over the set of right $J$-cosets, where $S_{n}$ is the symmetric group on $n$ elements. In particular, the core of $J$ is trivial if and only if the action mentioned above is faithful.
\end{proposition}
We leave the proof of this proposition to the reader. \\
We are going to construct an analogous notion of core at least for finite group-schemes that satisfies the same properties of the core of abstract subgroups. From now on, $k$ will be any field up to end of this subsection.
\begin{defi} \label{defi:Core_of_a_subgroup-scheme}
	Let $G$ be a finite group-scheme and $H \subset G$ a subgroup-scheme. As we mentioned in Remark \ref{remark:About_weak_quotients}, $A^{H}$ is the $k$-algebra associated to $G/H$ and the $G$-action over this scheme defines a comodule coaction of $A$ over $A^{H}$, or equivalently, a morphism of group-schemes $\mu:G \rightarrow \text{GL}(A^{H})$. We define the \textbf{core of $H$} in $G$ as $\text{Core}_{G}(H):=\ker(\mu)$.
\end{defi}
It is clear from the definition that the core of a subgroup-scheme $H$ is trivial if and only if $\mu$ is a faithful representation. We recall that as a functor, we have $\text{GL}(A^{H})(R)=\text{Aut}_{R}(A^{H} \otimes R)$.
\begin{proposition} \label{prop:Properties_of_the_core_of_subgroup-schemes}
	The core of a subgroup-scheme $H \subset G$ is the biggest normal subgroup-scheme of $G$ contained in $H$.
\end{proposition}
\proof
Let $F:\text{Alg}_{k}^{0} \rightarrow \text{Grp}$ be the functor given by 
$$F(R)=\{g \in \tilde{G}(R): \, \forall \text{ $R$-algebra } R \rightarrow S, \, g_{S} \in \text{Core}_{\tilde{G}(S)}(\tilde{H}(S))\}$$
where $g_{S}$ is the image of $g$ in $\tilde{G}(S)$, it is clearly a sub-functor of $\tilde{H}$ and for any $k$-algebra $R$, $F(R) \lhd \tilde{G}(R)$. Moreover, it contains any functor of the form $\tilde{N}$ where $N$ is a normal subgroup-scheme of $G$ that is contained in $H$ from Proposition \ref{prop:Properties_of_the_core_of_an_abstract_subgroup}. We will show that $F$ is the functor of points of $\text{Core}_{G}(H)$. \\
Let us denote the functor of points of $\text{Core}_{G}(H)$ as $\tilde{C}$, and let us start by showing that $\tilde{C} \subset F$. If $g_{R} \in \tilde{C}(R) \subset \tilde{G}(R)$ we can easily see that from the fact that it induces the identity $R$-automorphism of $A^{H} \otimes_{k} R$, this element acts like the identity over $\widetilde{G/H}(R)=\text{Hom}_{k-\text{alg}}(A^{H},R)$, in particular, it acts like the identity over $\tilde{G}(R)/\tilde{H}(R)$ and thus $g_{R} \in \text{Core}_{\tilde{G}(R)}(\tilde{H}(R))$ and likewise over any $R$-algebra $R \rightarrow S$, from which we conclude that $\tilde{C} \subset F$. \\
On the other hand, $F$ acts trivially over $\tilde{G}/\tilde{H}$ and then using the properties of $\widetilde{G/H}$ as a sheafification of this pre-sheaf, we see that the trivial action of $F$ over $\tilde{G}/\tilde{H}$ can be lifted to a trivial action of $F$ over $\widetilde{G/H}=\text{Hom}_{k-\text{alg}}(A^{H},-)$, this can be viewed as a natural transformation of functors $F \rightarrow \text{GL}(A^{H})$ that has a trivial image over any $k$-algebra. Thus, $F \subset \tilde{C}$ finishing the proof.
\endproof
\subsection{Pull-back of pure torsors} \label{subsec:Pull-back_of_pure_torsors}
Let us start with a lemma:
\begin{lemma} \label{lemma:Faithful_representations_over_non-pull-back_torsors}
	Let $f:X \rightarrow S$ be a morphism of proper, reduced and connected schemes of finite type over $k$, with $k$ perfect. We will further assume that the induced morphism $\pi^{N}(f):\pi_{1}^{N}(X,x) \rightarrow \pi_{1}^{N}(S,s)$ is faithfully flat for compatible rational points. \\
	Let $V$ be an essentially finite bundle corresponding to a representation of $\pi_{1}^{N}(X,x)$ that we suppose to lay inside a full subcategory $\text{Rep}_{k}(G)$ for a finite group-scheme $G$ such that the Nori-reduced $G$-torsor associated to the morphism of group-schemes $\pi_{1}^{N}(X,x) \rightarrow G$ is not the pull-back of a Nori-reduced torsor over $S$. Then, if $V$ corresponds to a faithful representation of $G$, it cannot be of the form $f^{*}(W)$ with $W \in EF(S)$.
\end{lemma}
\proof 
Let $\alpha:G \rightarrow \text{GL}(V_{x})$ be the representation morphism corresponding to $V$, where $V_{x}$ is the fiber of $W$ over $x$. As a representation of the FGS it can be seen as the composition $\pi_{1}^{N}(X,x) \rightarrow G \rightarrow \text{GL}(V_{x})$. The image of this composition is isomorphic to $G$. \\
Let us suppose that $V$ is the pull-back of an essentially finite bundle $W \in EF(S)$, this implies that we have a commutative diagram
$$\xymatrix{\pi_{1}^{N}(X,x) \ar[r] \ar[d]_{\pi^{N}(f)}& G \ar[r]^-{\alpha} & \text{GL}(V_{x})\\
            \pi_{1}^{N}(S,s) \ar[rru] & &} $$ 
as $\pi^{N}(f)$ is faithfully flat, $\pi_{1}^{N}(S,s) \rightarrow \text{GL}(V_{x})$ factors through $G$. From this, we can easily see that if $T \rightarrow X$ is the $G$-torsor associated to $\pi_{1}^{N}(X,x) \rightarrow G$, it is a pull-back from a $G$-torsor over $S$, a contradiction.
\endproof
We can apply this to weak quotients of pure torsors as follows:
\begin{lemma} \label{lemma:Weak_quotients_of_pure_torsors_are_not_pull-backs}
	Let $f:X \rightarrow S$ be a morphism of proper, reduced and connected schemes of finite type over $k$, where the induced morphism $\pi^{N}(f):\pi_{1}^{N}(X,x) \rightarrow \pi_{1}^{N}(S,s)$ is faithfully flat for compatible rational points. \\
 Let $t: T\rightarrow X$ be a pure Nori-reduced $G$-torsor and let $H$ be any non-trivial subgroup-scheme of $G$. Then for the essentially finite cover $t^{\prime}:T/H \rightarrow X$, the associated essentially finite bundle $\left(t^{\prime}\right)_{*}(\Strsh{T/H}) $ is not the pull-back of any essentially finite bundle over $S$.
\end{lemma}
\proof 
We will proceed by induction over $\text{ord}(G)$, the case of order 1 is trivial. Let us now suppose $\text{ord}(G) > 1$ and let $H$ be a subgroup-scheme of $G$. If $H$ is normal, there is nothing to prove as $T$ is pure. Otherwise, we consider the cover $t^{\prime}:T/H \rightarrow X$ corresponding in terms of essentially finite bundles to the sub-representation $A^{H}$ of $A$ where $G=\Spec(A)$. Let $K$ be the core of $H$ (Definition \ref{defi:Core_of_a_subgroup-scheme}), there are two possibilities for $K$:
\begin{itemize}
	\item If $K$ is not trivial, then we can see $T/H$ as a quotient of $T/K$ from Proposition \ref{prop:Properties_of_the_core_of_subgroup-schemes}, and thus result follows by applying the induction hypothesis.
	\item If $K$ is trivial, the representation corresponding to $\left(t^{\prime}\right)_{*}(\Strsh{T/H})$ is faithful and thus it is not a pull-back from Lemma \ref{lemma:Faithful_representations_over_non-pull-back_torsors}.
\end{itemize}
\endproof
From this lemma, from now on we can change the definition of a pure torsor in Definition \ref{definition:Pure_pull-back_and_mixed_torsor} to include any quotient, weak quotients included. \\
We need one last lemma in order to prove the main result of this chapter.
\begin{lemma} \label{lemma:Essentially_finite_bundles_trivial_over_the_geometric_generic_fiber_are_pull-backs}
	Let $f:X \rightarrow S$ be a faithfully flat and proper morphism of proper, reduced and connected schemes of finite type over $k$, with $k$ algebraically closed. We will further suppose that $S$ is integral and all geometric fibers are reduced and connected, thus they all have a FGS. \\
	If $V \in EF(X)$ is an essentially finite bundle whose restriction to the geometric generic fiber $X_{\bar{\eta}}$ is trivial where $\eta$ is the generic point of $S$, then $V$ is the pull-back of an essentially finite bundle over $S$. Moreover, if $V=t_{*}(\Strsh{T})$ where $t:T \rightarrow X$ is a finite Nori-reduced torsor, then $T$ is the pull-back of a Nori-reduced torsor over $S$.
\end{lemma}
\proof
Let $r \geq 1$ be the rank of $V$. From Remark \ref{remark:Trivial_ess_finite_bundles_in_terms_of_global_sections} we obtain that $h^{0}(X_{\bar{\eta}}, \left. V \right|_{X_{\bar{\eta}}})$ is equal to $r$. By the semi-continuity theorem, we have for any point $s \in S$ that 
$$h^{0}(X_{\bar{\eta}}, \left. V \right|_{X_{\bar{\eta}}}) \leq  h^{0}(X_{s}, \left. V \right|_{X_{s}})\leq r$$ 
and thus the rank of the 0-th cohomology group of $V$ is constant along the fibers of $S$. This implies that $f_{*}(V)$ is locally free over $S$ (see \cite[III Coro. 12.9]{Hartshorne}). \\
 If $W=f_{*}(V)$, arguing as in \cite[II Prop. 9]{NoriTFGS81}, we see that $f^{*}(W) \cong V$ and thus $W$ is essentially finite from Lemma \ref{lemma:If_the_pull-back_of_a_bundle_is_ess_finite_then_it_is_ess_finite} and the statement for the essentially finite bundles coming from a finite Nori-reduced torsor comes from \cite[II Prop. 9]{NoriTFGS81} as well.
\endproof
Now we state the main proposition of this chapter, $k$ will be algebraically closed:
\begin{proposition} \label{prop:Pull-back_of_pure_torsors_to_the_geometric_generic_fiber}
		Keeping the hypotheses of Lemma \ref{lemma:Essentially_finite_bundles_trivial_over_the_geometric_generic_fiber_are_pull-backs},
		let $t:T \rightarrow X$ be a NR pure $G$-torsor over $X$. Then, the pull-back $T_{\bar{\eta}} \rightarrow X_{\bar{\eta}}$ to the geometric generic fiber of $f$ is Nori-reduced, where $\bar{\eta}$ is the geometric generic point over $S$.
\end{proposition}
\proof
The pull-back torsor $T_{\bar{\eta}}$ of $T$ corresponds to a morphism of group-schemes $\pi_{1}^{N}(X_{\bar{\eta}}) \rightarrow G_{\bar{\eta}}$. Let $H_{\bar{\eta}} \subset G_{\bar{\eta}}$ be the image of this morphism and let us suppose that it is not equal to $G_{\bar{\eta}}$: if it is trivial, we will immediately get a contradiction from Lemma \ref{lemma:Essentially_finite_bundles_trivial_over_the_geometric_generic_fiber_are_pull-backs} and the pureness of $T$.\\ 
If $H_{\bar{\eta}}$ is a non-trivial subgroup-scheme of $G_{\bar{\eta}}$, then from good correspondence (Corollary \ref{corollary:Good_correspondence_of_subgroups_implies_good_correspondence_of_quotient_torsors}) both quotients $T_{\bar{\eta}}/H_{\bar{\eta}}$ and $T/H$ exist. \\
We can easily see that the essentially finite bundle associated to $t^{\prime}:T/H \rightarrow X$ has a trivial restriction to the geometric generic fiber, and thus from Lemma \ref{lemma:Essentially_finite_bundles_trivial_over_the_geometric_generic_fiber_are_pull-backs} this bundle is the pull-back of an essentially finite bundle over $S$, contradicting Lemma \ref{lemma:Weak_quotients_of_pure_torsors_are_not_pull-backs}.
\endproof
\begin{remark} \label{remark:FGS_of_weak_quotients}
	As communicated to the author by M. Emsalem, it can be shown that weak quotients possess fundamental group-schemes applying \cite[Theorem I (1)]{AntBiswEmsToniZhang2017}. This should extend the bijection between pointed $G$-torsors and arrows $\pi_{1}^{N}(X,x) \rightarrow G$ to weak quotients on the right and weak quotients of torsors, where weak quotients of torsors possess a FGS if and only if they are a quotient of a Nori-reduced torsor. \\
	With this, an alternative proof of Lemma \ref{lemma:Weak_quotients_of_pure_torsors_are_not_pull-backs} that does not require considering the core of a subgroup-scheme can be given. This result will be given in the author's Ph.D. thesis manuscript.
\end{remark}
 \section{Finiteness of the kernel} \label{sec:Finiteness_of_the_kernel}
\subsection{Setting and notation} \label{subsec:Setting_and_notation}
In this section we will lay down the setting that we will use in chapters \ref{sec:Finiteness_of_the_kernel} and \ref{sec:Base_change_and_proof_of_main_theorem} in order to prove Theorem \ref{theorem:Main_theorem}.\\ 
From now on $k$ will be an algebraically closed field. As such, we will often omit the rational points when writing the FGS, as different rational points give rise to isomorphic fundamental group-schemes \cite[II Prop. 4 (c)]{NoriTFGS81}. We will make exceptions when it becomes relevant, this also means that we can change rational points at whim as long as we have compatible ones. \\
Let $f:X \rightarrow S$ be a faithfully flat morphism, between a proper variety and an abelian variety $S$, we will denote its induced morphism at the level of FGS as $\pi^{N}(f): \pi_{1}^{N}(X) \rightarrow \pi_{1}^{N}(S)$. \\
We will further assume that all geometric fibers are reduced, connected and possess a finite FGS. This includes the geometric generic fiber $X_{\bar{\eta}}$ where $\eta$ is the generic point of $S$.\\ 
With these hypotheses, we have that $\pi^{N}(f)$ is faithfully flat by Proposition \ref{prop:If_the_morphism_of_S-FGSs_is_faithfully_flat_then_so_is_the_morphism_between_FGSs}. Another natural example where this holds, is the case when $f$ is smooth and proper (see one of the corollaries of \cite[II Prop. 6]{NoriTFGS81}).\\
For a $G$-torsor $T$ over $X$, we will denote $t:T \rightarrow X$ its structural morphism and $V=t_{*}(\Strsh{T})$ its associated essentially finite bundle. \\
In this chapter we are going to work with a particular pro-NR torsor, associated with the kernel of $\pi^{N}(f):\pi_{1}^{N}(X) \rightarrow \pi_{1}^{N}(S)$. 
\subsection{The universal pull-back torsor} \label{subsec:Universal_pull-back_torsor}
Before working with the kernel of $\pi^{N}(f)$, we will state a general remark. \\
For the sake of generality, let us suppose for this section that $k$ is just perfect $X$ and $S$ are proper, reduced and connected, and $f:X \rightarrow S$ is such that $\pi^{N}(f):\pi_{1}^{N}(X,x) \rightarrow \pi_{1}^{N}(S,s)$ is faithfully flat for compatible rational points.
\begin{remark} \label{remark:Remarks_on_torsor_types}
	Let $T$ be a Nori-reduced $G$-torsor over $X$, then we have the natural composition $\ker(\pi^{N}(f)) \rightarrow \pi_{1}^{N}(X) \rightarrow G$ whose image we will be denoted as $K$, this is a normal subgroup-scheme of $G$, and the relationship between this subgroup-scheme and $G$ determines the nature of $T$:
	\begin{itemize}
		\item If $K$ is a proper sub-group-scheme of $G$, the arrow $\pi_{1}^{N}(X) \rightarrow G/K$ factors through $\pi^{N}(f)$ and thus it corresponds to a $G/K$-torsor $t^{\prime}:X^{\prime} \rightarrow X$ that is the pull-back of a Nori-reduced torsor $p^{\prime}:S^{\prime} \rightarrow S$ over the same group-scheme, this pull-back quotient is the ``biggest'', in the sense that any other quotient of $T$ that is a pull-back is a quotient $X^{\prime}$. In the particular case when $K$ is trivial, $T$ itself is a pull-back. From now on, we will call the torsor $X^{\prime}$ the \textbf{maximal pull-back quotient of $T$}.\\
		Just to add up to our notations, as $T$ is also a torsor over $X^{\prime}$, we will denote its structural arrow as $q:T \rightarrow X^{\prime}$, this is a Nori-reduced $K$-torsor defined as the quotient scheme of $T$ by the action of $K$.
		\item From the latter point, we observe that a Nori-reduced torsor is pure if and only if $K=G$. 
		\item A mixed $G$-torsor $p:T \rightarrow X$ is certainly not pure with respect to $f$, but the torsor $q:T \rightarrow X^{\prime}$ is pure with respect to the base change of $f$ to $S^{\prime}$, $f^{\prime}: X^{\prime} \rightarrow S^{\prime}$.
	\end{itemize}
\end{remark}
We would like to further characterize the kernel of $\pi^{N}(f): \pi_{1}^{N}(X,x) \rightarrow \pi_{1}^{N}(S,s)$ using Proposition \ref{proposition:Pro-saturated_torsors_possess_a_FGS} to see it as a FGS. We will prove the scheme whose FGS is the kernel is the following:
\begin{defi} \label{defi:Universal_pull-back_torsor}
	Let $\hat{S}$ be the universal torsor of $S$, we define the \textbf{universal pull-back torsor} as the torsor $X^{*}:=\hat{S} \times_{S} X$ over $X$.
\end{defi}
We can clearly see that $X^{*}$ is the projective limit over $X$ of all the torsors of the form $T \times_{S} X$ where $T \rightarrow S$ is a finite Nori-reduced torsor over $S$. \\
When the induced morphism $\pi^{N}(f): \pi_{1}^{N}(X) \rightarrow \pi_{1}^{N}(S)$ is faithfully flat, we have the following result:
\begin{proposition} \label{prop:If_the_morphism_between_FGSs_is_faithfully_flat_then_the_universal_pull-back_has_a_FGS}
	Let $f:X \rightarrow S$ be a morphism between proper, reduced and connected schemes over a perfect field $k$, such that the induced morphism $\pi^{N}(f): \pi_{1}^{N}(X,x) \rightarrow \pi_{1}^{N}(S,s)$ for compatible rational points is faithfully flat. \\
	Then $X^{*}$ is a pro-NR torsor and in that case, we have $\pi_{1}^{N}(X^{*},x^{*}) = \ker(\pi^{N}(f))$ where $x^{*} \in X^{*}(k)$ is a rational point over $x$.
\end{proposition}
\begin{proof}
	Let $S_{i} \rightarrow S$ be a finite Nori-reduced torsor over $S$, then its pull-back $X_{i}:=S_{i} \times_{S} X$ over $X$ is Nori-reduced as $\pi^{N}(f)$ is faithfully flat. As we mentioned before, $X^{*}$ is the projective limit of the torsors $X_{i}$, thus it is pro-NR and it possesses a FGS from Proposition \ref{proposition:Pro-saturated_torsors_possess_a_FGS} that is $\ker(\pi^{N}(f))$.
\end{proof}
Using Lemma \ref{lemma:Torsors_over_pro-NR_torsors_descent_to_a_finite_torsor}, the main properties of this pro-NR torsor are the following:
\begin{proposition} \label{prop:Properties_of_the_universal_pull-back}
	We will denote by $\hat{S}=\displaystyle{ \lim_{\leftarrow} } \, S_{i}$ where the limit is taken over a directed set $I$ of indexes, and consequently $X^{*} = \displaystyle{\lim_{\leftarrow} } \, X_{i}$ ($i \in I$), where $X_{i}=S_{i} \times_{S} X$. We will also add, according to Section \ref{subsec:Notation_for_projective_limits_of_torsors}, an auxiliary index ``0'' such that $X_{0}:=X$. Let $T \rightarrow X$ be a Nori-reduced torsor, $T^{\prime}$ its pull-back to $X^{*}$ and let $V \rightarrow X^{*}$ be a finite Nori-reduced torsor over the universal pull-back torsor. Then:
	\begin{enumerate}
		\item If $T$ is pure, $T^{\prime}$ is Nori-reduced.
		\item For any index $i$ for which $V$ descends to a torsor $V_{i} \rightarrow X_{i}$ fitting into a cartesian diagram
		$$ \xymatrix{V \ar[r] \ar[d] & V_{i} \ar[d] \\
			X^{*} \ar[r] & X_{i}  } $$
		we have that the torsor $V_{i}$ is pure with respect to the morphism of schemes $f_{i}: X_{i} \rightarrow S_{i}$ where $f_{i}$ is the base change of $f$ via $S_{i} \rightarrow S$.
		\item If $V$ does not descend to a torsor over $X$. There exist a big enough index $j$ such that the descent of $V$ over $X_{j}$, that we denote by $V_{j}$, is the quotient of a mixed torsor over $X$, whose maximal pull-back quotient is $X_{j}$.
	\end{enumerate}  
\end{proposition}
\proof
(1) follows from the last remark. For (2), let $f_{i}:X_{i} \rightarrow S_{i}$ be the base change of $f$ that comes from $S_{i} \rightarrow S$, and let $\pi^{N}(f_{i}): \pi_{1}^{N}(X_{i}) \rightarrow \pi_{1}^{N}(S_{i})$ be the induced morphism between the corresponding FGS. We will show that $\ker(\pi^{N}(f_{i}))=\ker(\pi^{N}(f))$ which implies (2) also from the last remark. \\
Firstly, from the commutative diagram of group-schemes
$$ \xymatrix{\pi_{1}^{N}(X_{i}) \ar[r]^-{\pi^{N}(f_{i})} \ar[d] & \pi_{1}^{N}(S_{i}) \ar[d] \\
	\pi_{1}^{N}(X) \ar[r]^{\pi^{N}(f)} & \pi_{1}^{N}(S)} $$
where the vertical arrows are closed immersions, we get the inclusion $\ker(\pi^{N}(f_{i})) \subset \ker(\pi^{N}(f))$. And on the other hand, from the tower of torsors 
$$X^{*} \rightarrow X_{i} \rightarrow X$$ 
we also have a commutative diagram where all the arrows are inclusions of subgroup-schemes
$$ \xymatrix{\pi_{1}^{N}(X^{*}) = \ker(\pi^{N}(f)) \ar[r] \ar[rd] & \pi_{1}^{N}(X_{i}) \ar[d] \\
	& \pi_{1}^{N}(X)} $$
which implies, together with the diagram right above, that $\ker(\pi^{N}(f)) \subset \ker(\pi^{N}(f_{i}))$, finishing the proof. \\
Lastly, to prove (3), let us take one index $i$ with a descent $V_{i} \rightarrow X_{i}$ of $V \rightarrow X^{*}$. This gives us the tower of torsors $V_{i} \rightarrow X_{i} \rightarrow X$ and thus we can take its closure $W$, thus we have the following commutative diagram
$$ \xymatrix{V_{i} \ar[d] & \\
	X_{i} \ar[d] &W \ar[ul] \ar[l] \ar[ld] \\
	X &}. $$
All torsor forming the tower are Nori-reduced with respect to each corresponding base. This last fact implies that $V_{i}$ is a quotient of $W$ as a torsor over $X_{i}$. \\
We also see that $W$ is a mixed torsor over $X$, but as such its maximal pull-back quotient is not necessarily $X_{i}$. Let $W \rightarrow X_{j}$ be the maximal pull-back quotient of $W$, in that case we have that $j \geq i$ and we can use the arrow $W \rightarrow V_{i}$ to get an arrow to the fibered product $W \rightarrow V_{i} \times_{X_{i}}  X_{j}$ which is actually $V_{j}$, this arrow is a quotient of torsors, finishing the proof of (3).
\endproof
\subsection{Comparison of geometric generic pull-backs of torsors} \label{subsec:Eta-xi_lemma}
Going back to the assumptions of Section \ref{subsec:Setting_and_notation}. The fact that for $f:X \rightarrow S$, $S$ is an abelian variety, implies the following:
\begin{lemma} \label{lemma:Characterizing_the_induced_morphism_between_torsors_pull-back_case}
	Let $f:X \rightarrow S$ be as in Section \ref{subsec:Setting_and_notation}. Then, let $p^{\prime}:X^{\prime} \rightarrow X$ be a Nori-reduced $G$-torsor that is the pull-back of a Nori-reduced $G$-torsor $t^{\prime}:S^{\prime} \rightarrow S$, and let $f^{\prime}:X^{\prime} \rightarrow S^{\prime}$ be the base change to $S^{\prime}$ of $f$. Then, $X^{\prime}$ is projective, reduced and connected, $S^{\prime}$ is an abelian variety and $f^{\prime}$ is proper and smooth with reduced and connected geometric fibers.
\end{lemma}
\proof
As stated in Remark \ref{remark:Remarks_about_the_FGS_of_an_abelian_variety}, $S^{\prime}$ is an abelian variety.
As the properties of $f^{\prime}$ are inherited from those of $f$ because they are preserved by base change, we conclude that the proper scheme $X^{\prime}$ is also reduced and because $\Gamma(X^{\prime},\Strsh{X^{\prime}})=k$ we also conclude that $X^{\prime}$ is connected. \\
The last bit that we need to prove is that the geometric fibers of $f^{\prime}$ posses a FGS. To that end, we just need to prove that they are reduced and connected, but as $k$ is algebraically closed, we have the following property (\cite[Lemma 36.26.2 Tag 055C, Lemma 36.24.2 Tag 0574]{stacks-project}): If 
$$Y_{S}=\{s \in S: \, X_{s} \text{ is geometrically connected and reduced } \}$$
and $Y_{S^{\prime}}$ is its counterpart for $S^{\prime}$ and $X^{\prime}$, then we have that
$$ Y_{S^{\prime}} = \left( t^{\prime} \right)^{-1} \left( Y_{S} \right). $$
Thus, as $Y_{S}=S$ in our case and $t^{\prime}$ is surjective, then we conclude that $Y_{S^{\prime}}=S^{\prime}$, as we wanted.
\endproof
This lemma implies that Proposition \ref{prop:Pull-back_of_pure_torsors_to_the_geometric_generic_fiber} holds for $f^{\prime}$: If $\xi$ is the generic point of $S^{\prime}$, then a pure Nori-reduced torsor over $X^{\prime}$ has a Nori-reduced pull-back to the geometric generic fiber $X_{\bar{\xi}}^{\prime}$. \\
Now we will apply this to a special case: let $p:T \rightarrow X$ be a mixed Nori-reduced $G$-torsor, and let $p^{\prime}:X^{\prime} \rightarrow X$ be its maximal pull-back quotient where $t^{\prime}:S^{\prime} \rightarrow S$ is the $H$-torsor over $S$ whose pull-back to $X$ is $X^{\prime}$, if $q:T \rightarrow X^{\prime}$ is the quotient morphism, we have the following diagram with a cartesian square within
$$\xymatrix{T \ar[d]^{q} \ar@/_1pc/[dd]_-{p}& \\
	X^{\prime} \ar[r]^{f^{\prime}} \ar[d]^{p^{\prime}} & S^{\prime} \ar[d]^{t^{\prime}} \\
	X \ar[r]_{f} & S} $$
If $\eta$ and $\xi$ are the generic points of $S$ and $S^{\prime}$ respectively, we have the following comparison result for the pull-back of $T$ to $X_{\bar{\xi}}^{\prime}$:
\begin{lemma} \label{lemma:Lemma_eta-xi}
	Keeping the notations of the last paragraph, the pull-back $T_{\bar{\xi}}$ of $T$ to the geometric generic fiber $X_{\bar{\xi}}^{\prime}$ is the pull-back of a Nori-reduced torsor over $X_{\bar{\eta}}$.
\end{lemma}
\proof
First, we have a natural morphism of geometric generic points $\bar{\xi} \rightarrow \bar{\eta}$, and if we consider this morphism and the morphism $\bar{\xi} \rightarrow S^{\prime}$ we then have the following diagram
$$\xymatrix{\bar{\xi} \ar[r] \ar[dr] & S^{\prime}_{\bar{\eta}} \ar[r] \ar[d] & S^{\prime} \ar[d]^{t^{\prime}} \\
	& \bar{\eta} \ar[r]  &S  } $$
where the square is cartesian and the triangle is commutative. As $t^{\prime}:S^{\prime} \rightarrow S$ is a torsor, its geometric fiber is isomorphic to $H_{\bar{\eta}}$, the base change of the $k$-group-scheme $H$ to the residue field of $\bar{\eta}$. We can chose this isomorphism in such a way that the morphism $\bar{\xi} \rightarrow S_{\bar{\eta}}^{\prime}$ becomes the composition $\bar{\xi} \rightarrow \bar{\eta} \overset{\epsilon_{\bar{\eta}}}{\rightarrow} H_{\bar{\eta}}$ under this isomorphism where $\epsilon_{\bar{\eta}}$ is the unit rational point of the group-scheme $H_{\bar{\eta}}$. Now if we take the pull back of this diagram over $f:X \rightarrow S$, we will obtain the following commutative diagram
$$ \xymatrix{X \times_{S} \bar{\xi} \ar[r] \ar[dr] & X \times_{S} H_{\bar{\eta}} \ar[r] \ar[d] & X^{\prime} \ar[d]^{p^{\prime}} \\
	& X_{\bar{\eta}} \ar[r]  &X  }  $$
Notice that the product $X \times_{S} H_{\bar{\eta}}$ is 
$$ \begin{array}{rcl}
X \times_{S} H_{\bar{\eta}} & = & X \times_{S} \left( S^{\prime} \times_{S} \bar{\eta} \right) \\
& = & \left( X \times_{S} S^{\prime} \right) \times_{S} \bar{\eta} \\
& = & X^{\prime} \times_{S} \bar{\eta} \\
& = & X_{\bar{\eta}}^{\prime}.
\end{array} $$
In addition, If we call $X_{\bar{\xi}}$ the fibered product $X \times_{S} \bar{\xi}$, we see that
$$\begin{array}{rcl}
X \times_{S} \bar{\xi} & = & \left( X \times_{S} S^{\prime} \right) \times_{S^{\prime}} \bar{\xi} \\
& = & X^{\prime} \times_{S^{\prime}} \bar{\xi} \\
& = & X_{\bar{\xi}}^{\prime}
\end{array} $$
and then we see that $X_{\bar{\xi}}$ is the geometric generic fiber of $f^{\prime}: X^{\prime} \rightarrow S^{\prime}$. Since $X^{\prime}$ is a pull-back torsor, we have that $X_{\bar{\eta}}^{\prime}$ is a trivial torsor over $X_{\bar{\eta}}$ with a section $s_{\bar{\eta}}: X_{\bar{\eta}} \rightarrow X_{\bar{\eta}}^{\prime}$ fitting into the following diagram:
\begin{equation} \label{equation:Main_commutative_diagram_in_lemma_eta-xi}
\xymatrix{X_{\bar{\xi}}^{\prime} \ar[rr]^{\theta} \ar[rd]^{\lambda} & & X_{\bar{\eta}}^{\prime} \ar@/_/[dl]_{p_{\bar{\eta}}^{\prime}}  \\
	& X_{\bar{\eta}} \ar@/_/[ur]_{s_{\bar{\eta}}} &}.
\end{equation}
where the triangle made by $\theta$, $\lambda$ and $p_{\bar{\eta}}^{\prime}$ is commutative as well as the triangle made by the first two morphisms mentioned before, but with $s_{\bar{\eta}}$ instead of $p_{\bar{\eta}}^{\prime}$ as the third morphism.\\ 
Now let us consider $p:T \rightarrow X$ and take its pull-back $T_{\bar{\eta}}$ over $X_{\bar{\eta}}$ and the pull-back $T_{\bar{\xi}}$ over $X_{\bar{\xi}}^{\prime}$. From Proposition \ref{prop:Pull-back_of_pure_torsors_to_the_geometric_generic_fiber} we know that $T_{\bar{\xi}} \rightarrow X_{\bar{\xi}}^{\prime}$ is Nori-reduced as $q:T \rightarrow X^{\prime}$ is pure with respect to $f^{\prime}$ and it is worth to point out that $T_{\bar{\xi}}$ over $X_{\bar{\xi}}$ is the pull-back along $\theta$ of $q_{\bar{\eta}}:T_{\bar{\eta}} \rightarrow X_{\bar{\eta}}^{\prime}$, which is the pull-back to $X_{\bar{\eta}}$ of $q$. Let $z:Z \rightarrow X_{\bar{\eta}}$ be the pull-back of the torsor $q_{\bar{\eta}}:T_{\bar{\eta}} \rightarrow X_{\bar{\eta}}^{\prime}$ along $s_{\bar{\eta}}$. As $H$ is a quotient of $G$, if we call $K$ the normal subgroup-scheme of $G$ such that $H = G/K$, we have that $Z$ is $K$-torsor and from the last commutative square, we see that its pull-back along $\lambda$ is $T_{\bar{\xi}}$, but this pull-back is Nori-reduced, and then so is $Z$ which is then the torsor over $X_{\bar{\eta}}$ we were looking for, finishing the proof.
\endproof
Another consequence of the commutative diagram \ref{equation:Main_commutative_diagram_in_lemma_eta-xi} in the proof Lemma \ref{lemma:Lemma_eta-xi} is the following:
\begin{corollary} \label{coro:The_maximal_sub-torsor_of_a_mixed_torsor_over_the_geometric_generic_fiber_is_the_biggest_one_possible}
	Let $p:T \rightarrow X$ be a mixed $G$-torsor with maximal pull-back quotient $p^{\prime}:X^{\prime} \rightarrow X$, if $H=G/K$ for a certain normal subgroup-scheme $K$ of $G$ which is the group-scheme associated to $q^{\prime}:T \rightarrow X^{\prime}$. Then, the maximal Nori-reduced sub-torsor of $p_{\bar{\eta}}:T_{\bar{\eta}} \rightarrow X_{\bar{\eta}}$ is a $K$-torsor.
\end{corollary}
\proof
Firstly, as $X^{\prime}$ becomes trivial when taking the pull-back over $X_{\bar{\eta}}$, we have that the composition
$$ \pi_{1}^{N}(X_{\bar{\eta}}) \rightarrow G \rightarrow G/K $$
is trivial, and thus the image of the first arrow is contained in $K$. Let $V \subset X_{\bar{\eta}}$ be the maximal Nori-reduced sub-torsor of $X_{\bar{\eta}}$ which is not Nori-reduced, as its associated group-scheme corresponds to the image of $\pi_{1}^{N}(X_{\bar{\eta}}) \rightarrow G$, we see that $K$ is the biggest possible subgroup-scheme that could be associated to $V$. \\
Moreover, let us consider the Nori-reduced $K$-torsor $z:Z \rightarrow X_{\bar{\eta}}$ from the last proof, we will show that it is a sub-torsor of $T_{\bar{\eta}}$ which implies that $V=Z$. We have the following cartesian diagram
$$ \xymatrix{Z \ar[r]^{i} \ar[d]_{z} & T_{\bar{\eta}} \ar[d]^{q_{\bar{\eta}}} \\
	X_{\bar{\eta}} \ar[r]_{s_{\bar{\eta}} }& X_{\bar{\eta}}^{\prime} }  $$
and because $s_{\bar{\eta}}$ is a closed immersion, so is $i$, and the only thing we need to show to get that $Z$ is a sub-torsor, is the equality $p_{\bar{\eta}} \circ i = z$. From the cartesian diagram we see that $s_{\bar{\eta}} \circ z = q_{\bar{\eta}} \circ i$ and if we compose this equality with $p_{\bar{\eta}}^{\prime}$ we obtain:
$$\begin{array}{rcl}
\underbrace{p_{\bar{\eta}}^{\prime} \circ q_{\bar{\eta}}}_{=p_{\bar{\eta}}} \circ i & = & \underbrace{p_{\bar{\eta}}^{\prime} \circ s_{\bar{\eta}}}_{=id_{X_{\bar{\eta}}} } \circ z \\
p_{\bar{\eta}} \circ i & = & z
\end{array}$$
and effectively, $Z$ is a sub-torsor of $X_{\bar{\eta}}$ finishing the proof.
\endproof
\subsection{Proof of finiteness and consequences} \label{subsec:Proof_of_finiteness_for_the_kernel}
Now we will prove that for the fibration $f:X \rightarrow S$, the kernel of the induced morphism of fundamental group-schemes is finite as the FGS of $X_{\bar{\eta}}$ is finite and $S$ is an abelian variety. But first, we need a general lemma:
\begin{lemma} \label{lemma:If_a_scheme_over_k_has_a_FGS_where_orders_are_bounded_the_FGS_is_finite}
	Let $k$ be any field, and let $X$ be a $k$-scheme possessing a FGS with respect to the rational point $x \in X(k)$. Let us suppose that for any Nori-reduced $G$-torsor over $X$, the order of $G$ is bounded by a fixed finite positive integer, then $\pi_{1}^{N}(X,x)$ is finite. 
\end{lemma}
\proof
Let $M \in \mathbb{N}$ be the bound for the orders of the group-schemes acting over all Nori-reduced torsors over $X$. \\
Let $\mathcal{M}$ be the set of isomorphism classes of finite Nori-reduced torsors over $X$, this set comes with a natural partial ordering $V \leq V^{\prime}$ iff there exists a faithfully flat morphism $V^{\prime} \rightarrow V$ between representatives over $X$. In this case, we can use Zorn's lemma to get a maximal element of $\mathcal{M}$ that will be a finite Nori-reduced $G$-torsor over $X$ with $\text{ord}(G) \leq M$. \\
Let $\{V_{i}\}_{i \in I}$ be a chain of elements of $\mathcal{M}$. If the index set $I$ is finite, we can index the elements of the chain as $\{V_{i}\}_{i = 1} ^{n}$ and we have a chain of finite group-schemes 
$$ G_{1} \leftarrow G_{2} \leftarrow \cdots \leftarrow G_{n}  $$
with faithfully flat arrows between them, that corresponds to a chain of inclusions of Hopf-algebras
$$ A_{1} \hookrightarrow A_{2} \hookrightarrow \cdots \hookrightarrow A_{n} $$
and all the Hopf-algebras on the chain have finite $k$-dimension, bounded by $M$, and thus $V_{n}$ is a maximal element for the chain of torsors. We note as well that $A_{n}$ is isomorphic to the directed limit of the chain of Hopf-algebras. \\
If $I$ is infinite, we will get a similar chain of Hopf-algebras
$$ A_{0} \hookrightarrow A_{1} \hookrightarrow \cdots $$
and each Hopf-algebra $A_{i}$ ($i \in I$) has finite $k$-dimension bounded by $M$. As the $k$-dimension of the Hopf-algebras on the chain increases, this chain must be eventually stationary, i.e., $A_{i} \cong A_{j}$ for $i\leq j$, after a finite amount of inclusions starting from $A_{i_{0}}$, if $A_{N}$ is the least element that stabilizes the chain of Hopf-algebras, we see that the direct limit $\displaystyle{\lim_{\rightarrow}} A_{i}$ is isomorphic to $A_{N}$ and thus to all $A_{i}$ with $i \geq N$, this means that we can take $V_{N}$ as the maximal element of the chain, and the dimension of $A_{N}$ is bounded by $M$. \\
As we have assured the existence of finite a $G$-torsor $U \rightarrow X$ ($\text{ord}(G) \leq M$), the supremum of $\mathcal{M}$, we can see that for any finite Nori-reduced torsor $T$ over $X$, we have a faithfully flat arrow $U \rightarrow X$. We will conclude by showing that $U$ is unique modulo isomorphism. Let $V$ be another supremum of $\mathcal{M}$, then we can take the fibered product $U \times_{X} V$. This torsor might not be Nori-reduced but it has a Nori-reduced sub-torsor $Z \subset U \times_{X} V$ with faithfully flat arrows $Z \rightarrow U$ and $Z \rightarrow V$ as both $U$ and $V$ are Nori-reduced, this forces $Z \cong U \cong V$. Thus $U$ is unique under isomorphism, and $U \cong \hat{X}$ finishing the proof.
\endproof
\begin{proposition} \label{proposition:Finiteness_of_the_kernel}
	Let $f:X \rightarrow S$ be as in section \ref{subsec:Setting_and_notation}. Then, the kernel of the induced morphism $\pi^{N}(f):\pi_{1}^{N}(X) \rightarrow \pi_{1}^{N}(S)$ is finite.
\end{proposition}
\proof
Using Lemma \ref{lemma:If_a_scheme_over_k_has_a_FGS_where_orders_are_bounded_the_FGS_is_finite}, we will show that the order of the group-scheme associated to a finite Nori-reduced torsor $V \rightarrow X^{*}$ is bounded by a fixed positive integer. \\
Keeping the notation of Proposition \ref{prop:Properties_of_the_universal_pull-back}, there are two cases:
\begin{itemize}
	\item[(a)] If $V$ descends to a finite Nori-reduced and pure torsor $V_{0}$ over $X$, then we see that the pull-back to $X_{\bar{\eta}}$ is Nori-reduced by Proposition \ref{prop:Pull-back_of_pure_torsors_to_the_geometric_generic_fiber} and thus the order of the group-scheme associated to $V$ is bounded by $|\pi_{1}^{N}(X_{\bar{\eta}})|$, the order of $\pi_{1}^{N}(X_{\bar{\eta}})$.
	\item[(b)] If $V$ descends to a finite torsor $V_{i} \rightarrow X_{i}$, that is pure with respect to $f_{i}:X_{i} \rightarrow S_{i}$, by Proposition \ref{prop:Properties_of_the_universal_pull-back}(3), we can suppose that there is a mixed torsor $T \rightarrow X$ such that $X_{i}$ is its maximal pull-back quotient and $V_{i}$ is a quotient of $T$ over $X_{i}$. From Lemma \ref{lemma:Lemma_eta-xi} and Corollary \ref{coro:The_maximal_sub-torsor_of_a_mixed_torsor_over_the_geometric_generic_fiber_is_the_biggest_one_possible}, the order of $T \rightarrow X_{i}$ is the order of the maximal Nori-reduced sub-torsor of the pull-back $T_{\bar{\eta}}$ over $X_{\bar{\eta}}$, which is then bounded by $|\pi_{1}^{N}(X_{\bar{\eta}})|$, finishing the proof.
\end{itemize}
\endproof
Finally, we outline a strong consequence of the finiteness of the kernel.
\begin{corollary} \label{coro:Consequence_of_a_finite_kernel}
	Let $f:X \rightarrow S$ be a morphism between proper, connected and reduced $k$-schemes of finite type with $k$ perfect, and let us suppose that the induced morphism (for compatible rational points) $\pi^{N}(f):\pi_{1}^{N}(X,x) \rightarrow \pi_{1}^{N}(S,s)$ is faithfully flat with finite kernel. Then, there is a Nori-reduced torsor $X_{i} \rightarrow X$ that is the pull-back of a Nori-reduced torsor $S_{i} \rightarrow S$, such that we have an isomorphism $\pi_{1}^{N}(X_{i},x_{i}) \cong \pi_{1}^{N}(S_{i},s_{i}) \times_{k} \ker(\pi^{N}(f))$ for compatible rational points coming from the morphism $X_{i} \rightarrow S_{i}$.
\end{corollary}
\proof
As the kernel of $\pi^{N}(f)$ is finite, the torsor $\hat{X} \rightarrow X^{*}$ is finite. Using Lemma \ref{lemma:Torsors_over_pro-NR_torsors_descent_to_a_finite_torsor} as $X^{*}$ is pro-NR, $\hat{X} \rightarrow X^{*}$ descends to a finite torsor $\hat{X}_{i} \rightarrow X_{i}$ where $X_{i} \rightarrow X$ is the pull-back of a Nori-reduced torsor $S_{i} \rightarrow S$ such that we have the following cartesian diagram
$$\xymatrix{\hat{X} \ar[r] \ar[d] & \hat{X}_{i} \ar[d] \\
	X^{*} \ar[r]& X_{i} }. $$
In particular, we have that $\hat{X} = X^{*} \times_{X_{i}} \hat{X}_{i}$ and this implies for $x_{i} \in X_{i}(k)$ that $\pi_{1}^{N}(X_{i},x_{i}) = G \times \ker(\pi^{N}(f))$ is the product of the group-schemes corresponding to $X^{*}$ and $\hat{X}_{i}$ over $X_{i}$. To characterize $G$ as $\pi_{1}^{N}(S_{i},s^{\prime})$, we see from the commutative diagram
$$ \xymatrix{X^{*} \ar[r] \ar[d] & \hat{S} \ar[d] \\
	X_{i} \ar[r]& S_{i} } $$
that $\hat{S} \rightarrow S_{i}$ and thus $X^{*} \rightarrow X_{i}$ is a $\pi_{1}^{N}(S_{i},s_{i})$-torsor for some $s_{i} \in S_{i}(k)$ compatible with $x_{i}$, finishing the proof.
\endproof
\begin{remark} \label{remark:What_if_the_kernel_is_finite_and_the_universal_torsor_descends_to_X}
	In the case the torsor $\hat{X} \rightarrow X^{*}$ descends to a torsor over $X$, we will obtain that $\pi_{1}^{N}(X,x) = \pi_{1}^{N}(S,s) \times_{k} \ker(\pi^{N}(f))$ and thus we would obtain a split short exact sequence
	$$ 1 \rightarrow \ker(\pi^{N}(f)) \rightarrow \pi_{1}^{N}(X,x) \rightarrow \pi_{1}^{N}(S,s) \rightarrow 1 .$$
\end{remark}
\endproof
\section{Base change condition and proof of the exact sequence} \label{sec:Base_change_and_proof_of_main_theorem}
Keeping the hypotheses and the notation outlined in \ref{subsec:Setting_and_notation}, we will restate the main theorem:
\begin{theorem} \label{theorem:Main_theorem}
	Let $f:X \rightarrow S$ be a fibration as in Section \ref{subsec:Setting_and_notation}, with $k$ an uncountable algebraically closed field. 
	Then, there exists a rational point $s \in S(k)$ such that the following exact sequence:
	$$ \pi_{1}^{N}(X_{s},x) \rightarrow \pi_{1}^{N}(X,x) \rightarrow \pi_{1}^{N}(S,s) \rightarrow 1  $$
	is exact where $x \in X(k)$ is a rational point of $X$ over $s$.
\end{theorem}
For the proof of this result, we will verify one of the equivalent conditions for the exact sequence above outlined by L. Zhang in \cite{ZhangL13EGF}, these are the following: 
\begin{theorem}[\cite{ZhangL13EGF}] \label{teo:Homotopy_exact_sequence_conditions}
	Let $f:X \rightarrow S$ be a proper morphism, with reduced and connected geometric fibers, between two reduced and connected locally noetherian schemes over a perfect field $k$. We additionally suppose that $S$ is irreducible and we take $x \in X(k)$ and $s \in S(k)$ such that $f(x)=s$. Then the following statements are equivalent:
	\begin{enumerate}
		\item The sequence
		$$ \pi_{1}^{N}(X_{s},x) \rightarrow \pi_{1}^{N}(X,x) \rightarrow \pi_{1}^{N}(S,s) \rightarrow 1$$
		is exact.
		\item For any Nori-reduced $G$-torsor $t:T \rightarrow X$ with $G$ finite, the vector bundle $t_{*}(\Strsh{T})$ satisfies the base change condition at $s$ and the image of the composition $\pi_{1}^{N}(X_{s},x) \rightarrow \pi_{1}^{N}(X,x) \rightarrow G$ is a normal subgroup-scheme of $G$.
		\item For any Nori-reduced $G$-torsor $t:T \rightarrow X$ with $G$ finite, the vector bundle $t_{*}(\Strsh{T})$ satisfies the base change condition at $s$ and there exists a Nori-reduced $G^{\prime}$-torsor $t^{\prime}:T^{\prime} \rightarrow S$ with an equivariant morphism $\theta:T \rightarrow T^{\prime}$ such that the induced map $\left(t^{\prime}\right)_{*}(\Strsh{T^{\prime}})_{s} \rightarrow f_{*}\left(t_{*}(\Strsh{T})\right)_{s}$ of fibers over $s$ coming from $\theta$ is an isomorphism. 
	\end{enumerate} 
	Moreover, if $X$ and $S$ are proper, then we can add an additional equivalent condition. Namely:
	\begin{itemize}
		\item[(4)] For any Nori-reduced $G$-torsor $t:T \rightarrow X$ with $G$ finite, the vector bundle $t_{*}(\Strsh{T})$ satisfies the base change condition at $s$ and $f_{*}\left(t_{*}(\Strsh{T})\right)$ is essentially finite over $S$.
	\end{itemize}
\end{theorem}
We see that all of these conditions require a condition over the essentially finite bundle $V=t_{*}(\Strsh{T})$ for all Nori-reduced torsors $t:T \rightarrow X$ known as the ``base change at $s$'' for $s \in S(k)$ and another additional condition. In the next two sections we will define the base change condition, and verify it for our fibration together with one of Zhang's equivalent conditions for the homotopy exact sequence, after choosing a suitable rational point of $S$.
\subsection{The base change condition} \label{subsec:Base_change_condition}
We will begin by defining this condition.
\begin{definition} \label{definition:Base_change_condition}
	Let $f:X \rightarrow S$ be a map of schemes, and $\mathcal{F}$ a coherent sheaf of $\Strsh{X}$-modules. If $s: \Spec(\kappa(s)) \rightarrow S$ is a point, then we have the cartesian diagram:
	$$
	\xymatrix{X_{s} \ar[d]^{g} \ar[r]^{i}  & X \ar[d]^{f} \\
		\Spec(\kappa(s)) \ar[r]_{s} & S}.
	$$
	We say that $\mathcal{F}$ \textbf{satisfies the base change condition} at $s$ if the canonical map
	$$
	s^{*}\left(f_{*}(\mathcal{F}) \right) \rightarrow g_{*}\left( i^{*}(\mathcal{F}) \right)
	$$
	is surjective.
\end{definition}
\begin{remark} \label{remark:Remarks_on_base_change_condition}
	If $f$ is proper, $S$ is locally noetherian, $\mathcal{F}$ is coherent and flat over $S$, then $\mathcal{F}$ satisfies the base change condition at $s$ if and only if the canonical map above is an isomorphism (see \cite[III Thm. 12.11]{Hartshorne}). \\
	The base change condition over a point is a particular case of the general base change condition about the following cartesian diagram:
	$$
	\xymatrix{X^{\prime} \ar[d]^{g} \ar[r]^{v}  & X \ar[d]^{f} \\
		S^{\prime} \ar[r]_{u} & S}.
	$$
	and the canonical arrow
	\begin{equation} \label{equation:General_base_change_condition}
	u^{*}\left(f_{*}(\mathcal{F}) \right) \rightarrow g_{*}\left( v^{*}(\mathcal{F}) \right)
	\end{equation}
	for a quasi-coherent sheaf $\mathcal{F}$ of $X$. For the condition to hold, we demand that the former arrow to be surjective. \\
	Under certain assumptions, this condition holds for a wide family of quasi-coherent sheaves:
	\begin{itemize} 
		\item If $f$ is separated, of finite type and $u:S^{\prime} \rightarrow S$ is a flat morphism of noetherian schemes the arrow \ref{equation:General_base_change_condition} is an isomorphism for all quasi-coherent sheaves \cite[III Prop. 9.3]{Hartshorne}.
		\item Equation \ref{equation:General_base_change_condition} is also an isomorphism if $f$ is affine for all quasi-coherent sheaves of $\Strsh{X}$-modules and for any $u$ (\cite[Lemma 02KG]{stacks-project}). 
	\end{itemize}
\end{remark} 
For our particular case, the base change condition over a point $s$ of $S$ means that we have to prove that the following morphism of vector spaces
\begin{equation} \label{equation:Base_change_condition_on_points_unpacked}
f_{*}(\mathcal{F})_{s} \otimes_{\Strsh{S,s}} \kappa(s) \rightarrow \Gamma(X_{s},\left. \mathcal{F} \right|_{X_{s}})
\end{equation}
is surjective, and thus an isomorphism. \\
Also, as the generic point $\eta:\Spec{K} \rightarrow S$ of $S$ is a flat morphism, the base change condition is always generically satisfied. This implies for coherent sheaves on $X$ that there is an open set $U_{\mathcal{F}}$, containing $\eta$, such that for all closed points $s \in U_{\mathcal{F}}$, the base change condition is also satisfied at $s$ \cite[III Thm. 12.11 (a)]{Hartshorne}. \\
Now let us take $f:X \rightarrow S$ as in section \ref{subsec:Setting_and_notation} and let $t:T \rightarrow X$ be a Nori-reduced pointed $G$-torsor over $X$, corresponding to a faithfully flat arrow $\pi_{1}^{N}(X) \rightarrow G$, we need to prove that the vector bundle $V=t_{*}(\Strsh{T})$ satisfies the base change condition over a certain rational point of $S$. \\
We start with pull-back torsors:
\begin{proposition} \label{prop:Base_change_for_pull-backs}
	Let $t^{\prime}:X^{\prime} \rightarrow X$ be a $G$-torsor over $X$. Assume that $X^{\prime}$ is the pull-back of a $G$-torsor $p:S^{\prime} \rightarrow S$, then it satisfies the base change condition for all rational points of $S$.
\end{proposition}
\proof  From our hypotheses, we have the following cartesian diagram
$$\xymatrix{X^{\prime} \ar[d]_{t^{\prime}} \ar[r]^{f^{\prime}} & S^{\prime} \ar[d]^{p^{\prime}} \\
	X \ar[r]^{f} & S}. $$
As the fibers for both $f$ and $f^{\prime}$ are connected, we have that $f_{*}(\Strsh{X})=\Strsh{S}$ and $f_{*}^{\prime}(\Strsh{X}^{\prime})=\Strsh{S}^{\prime}$ and thus we have that $f_{*}(t_{*}^{\prime}(\Strsh{X^{\prime}} )=p_{*}^{\prime}(\Strsh{S^{\prime}})$ is a vector bundle on $S$ and thus it has a constant rank throughout the points of $S$. Let $r$ be the rank of $t_{*}^{\prime}(\Strsh{X^{\prime}})$, from the last equality $f_{*}(t_{*}^{\prime}(\Strsh{X^{\prime}}))$ has the same rank as it is equal to the order of $G$ for both bundles. \\
On the other hand, as the pull-back of $X^{\prime}$ to $X_{s}$ is trivial for all points $s \in S$, we have that $r=\Gamma(X_{s},\left. t_{*}^{\prime}(\Strsh{X^{\prime}}) \right|_{X_{s}})$ which finishes the proof as we have an isomorphism on the equation \ref{equation:Base_change_condition_on_points_unpacked}.
\endproof
Now we focus our attention to pure torsors:
\begin{proposition} \label{prop:Base_change_for_pure_torsors}
	Let $t:T \rightarrow X$ be a pure Nori-reduced torsor over $X$, there exists an open set $U_{T} \subset S$, containing the generic point of $S$, such that for any rational point $s \in U_{T}$, the pull-back of $T$ to the fiber $X_{s}$ is Nori-reduced and $V=t_{*}(\Strsh{T})$ satisfies the base change condition at $s$. 
\end{proposition}
\proof
Let $t:T \rightarrow X$ be a pure Nori-reduced torsor and $V=t_{*}(\Strsh{T})$ as we pointed out before, there exists an open set $U \subset S$ that contains the generic point of $S$ in which the base change condition for $V$ is satisfied at all of the rational points contained within. \\
Moreover, as the pull-back of $T$ to the geometric generic fiber $X_{\bar{\eta}}$ is Nori-reduced by Proposition \ref{prop:Pull-back_of_pure_torsors_to_the_geometric_generic_fiber} and thus $h^{0}(X_{\bar{\eta}},\left.V \right|_{X_{\bar{\eta}}})=1$ and by the semi-continuity theorem \cite[III Thm. 12.8]{Hartshorne}, there exist an open set $U^{\prime} \subset S$ such that $h^{0}(X_{\bar{\eta}},\left.V \right|_{X_{\bar{\eta}}})=1$ as well, which implies that the pull-back of $T$ to $X_{s}$ is Nori-reduced, for $s \in U^{\prime}(k)$.\\
Finally, by taking $U_{T}=U \cap U^{\prime}$ we finish the proof.
\endproof
Now the only Nori-reduced torsors remaining are the mixed ones:
\begin{proposition} \label{prop:Base_change_for_mixed_torsors}
	Let $t:T \rightarrow X$ be a mixed Nori-reduced $G$-torsor. If we write $H=G/K$ where $H$ is the group-scheme corresponding to the maximal pull-back quotient of $T$, there exists an open set $U_{T} \subset S$, containing the generic point of $S$, such that for any $s \in U_{T}(k)$, the maximal Nori-reduced sub-torsor of the pull-back of $T$ to $X_{s}$ is a $K$-torsor, and $V=t_{*}(\Strsh{T})$ satisfies the base change condition at $s$.
\end{proposition}
\proof
There exists an open set $U \subset S$ where $V$ satisfies the base change condition at any rational point of $U$. \\
Moreover, as the maximal Nori-reduced sub-torsor of the pull-back $T_{\bar{\eta}}$ of $T$ to the geometric generic fiber $X_{\bar{\eta}}$ is a $K$-torsor by Corollary \ref{coro:The_maximal_sub-torsor_of_a_mixed_torsor_over_the_geometric_generic_fiber_is_the_biggest_one_possible}, as $K$ is the image of the morphism $\pi_{1}^{N}(X_{\bar{\eta}}) \rightarrow G$. \\
Applying Corollary \ref{coro:Global_sections_of_bundles_coming_from_torsors}, we have that $h^{0}(X_{\bar{\eta}},\left. V \right|_{X_{\bar{\eta}}}) = r$ where $r = \text{ord}(H)$, and for any $s \in S(k)$ we have that $h^{0}(X_{s},\left. V \right|_{X_{s}}) \geq r$ as the image of $\pi_{1}^{N}(X_{s}) \rightarrow G$ corresponding to the pull-back of $T$ to $X_{s}$ has its image contained in $K$. Thus, by semi-continuity, there exists an open set $U^{\prime}$ where the maximal Nori-reduced sub-torsor of pull-back torsor $T_{s} \rightarrow X_{s}$ is a $K$-torsor for any $s \in U^{\prime}(k)$.  \\
Finally, by taking $U_{T}=U \cap U^{\prime}$ we conclude the proof.
\endproof
\subsection{Proof of the exact sequence} \label{subsec:Proof_of_main_theorem}
We are ready to finish the proof of Theorem \ref{theorem:Main_theorem}. We start with a remark: 
\begin{remark} \label{remark:Zhangs_conditions_satisfied_generically}
	From generic flatness, good correspondence (Proposition \ref{prop:Good_correspondence_for_finite_group_schemes} and Corollary \ref{corollary:Good_correspondence_of_subgroups_implies_good_correspondence_of_quotient_torsors}), and Corollary \ref{coro:The_maximal_sub-torsor_of_a_mixed_torsor_over_the_geometric_generic_fiber_is_the_biggest_one_possible} when $k$ is algebraically closed, we can colloquially state that, under the right hypotheses, ``\emph{one of Zhang's conditions for the homotopy exact sequence is generically satisfied}'' in the following sense: for any Nori-reduced $G$-torsor $t:T \rightarrow X$ with $G$ finite, the vector bundle $t_{*}(\Strsh{T})$ satisfies the base change condition at $\bar{\eta}$ and the arrow $\pi_{1}^{N}(X_{\bar{\eta}}) \rightarrow G_{\bar{\eta}}$ has normal image. This condition is comparable to Theorem \ref{teo:Homotopy_exact_sequence_conditions} (2). \\
	This implies that for any Nori-reduced torsor $t:T \rightarrow X$, at least one of Zhang's conditions can be satisfied over an open subset of $S$, dependent on $T$, as we have seen in the last section. The hard part of showing the the homotopy exact sequence in this case, is to find a rational point $s \in S(k)$ for which we can verify any of Zhang's conditions for \underline{all} Nori-reduced torsors over $X$. 
\end{remark}
 Now let $f:X \rightarrow S$ be as in section \ref{subsec:Setting_and_notation}, first we recall that the induced morphism $\pi^{N}(f):\pi_{1}^{N}(X) \rightarrow \pi_{1}^{N}(S)$ is faithfully flat (Proposition \ref{prop:If_the_morphism_of_S-FGSs_is_faithfully_flat_then_so_is_the_morphism_between_FGSs}), and secondly we notice that $\ker(\pi^{N}(f))$ is finite from Proposition \ref{proposition:Finiteness_of_the_kernel}. Then, from Corollary \ref{coro:Consequence_of_a_finite_kernel}, we can deduce the following lemma:
\begin{lemma} \label{lemma:Homotopy_exact_sequence_when_the_kernels_descends_to_the_base}
	Let $f:X \rightarrow S$ be as in section \ref{subsec:Setting_and_notation}. Let us suppose that $k=\bar{k}$ is an uncountable field, the torsor $\hat{X} \rightarrow X^{*}$ (Definition \ref{defi:Universal_pull-back_torsor}) is finite and that it descends to a finite torsor $\hat{X}_{0}$ over $X$. \\
	Then, there exists $s \in S(k)$ such that for compatible points, the following sequence
	$$ \pi_{1}^{N}(X_{s},x) \rightarrow \pi_{1}^{N}(X,x) \rightarrow \pi_{1}^{N}(S,s) \rightarrow 1 $$
	is exact.
\end{lemma}
\proof
Let $t:T \rightarrow X$ be a pure Nori-reduced torsor, as the pull-back $T \times_{X} X^{*}$ is Nori-reduced (Proposition \ref{prop:Properties_of_the_universal_pull-back}(1)), we have a faithfully flat morphism of torsors $\hat{X} \rightarrow T \times_{X} X^{*}$ over $X^{*}$ and thus a faithfully flat morphism $\hat{X}_{0} \rightarrow T$ over $X$. As both torsors are Nori-reduced, $T$ must be a quotient of $\hat{X}_{0}$ by a normal subgroup-scheme of $\ker(\pi^{N}(f))$ and thus there is a finite amount of isomorphism classes of pure Nori-reduced torsors in this case. This allows us to consider the open set $U_{P}=\bigcap_{T \in \mathcal{P}} U_{T}$ where $\mathcal{P}$ is the finite family of isomorphism classes of pure Nori-reduced torsors over $X$, and $U_{T}$ is the open set defined in Proposition \ref{prop:Base_change_for_pure_torsors}. The finiteness of $\mathcal{P}$ implies that $U_{P}$ is a dense open set of $S$ as it contains its generic point and it has rational points inside. Over any of such rational point $p \in U_{P}$, all pure Nori-reduced torsors satisfy the base change condition at $p$ and their pull-backs to the fiber $X_{p}$ are Nori-reduced. \\
If $t:T \rightarrow X$ is mixed and Nori-reduced, assuming the notation of Proposition \ref{prop:Properties_of_the_universal_pull-back}, if $X_{i}:=S_{i} \times_{S} X$ is its maximal pull-back quotient where $S_{i} \rightarrow S$ is Nori-reduced, by using a similar argument to the one we used for pure torsors, we see that $T \rightarrow X_{i}$ is a quotient torsor of the pure torsor $\hat{X}_{i} \rightarrow X_{i}$, the descent of $\hat{X} \rightarrow X^{*}$ over $X_{i}$, with respect to $f_{i}:X_{i} \rightarrow S_{i}$. This implies there is a finite amount of classes of isomorphic pure torsors over $X_{i}$. As $S$ is an abelian variety, it possesses a countable amount of isomorphism classes of Nori-reduced torsors and thus there is a countable amount of isomorphism classes of pull-back torsors $X_{i}$, so we conclude that there is a countable amount of isomorphism classes of mixed Nori-reduced torsors over $X$. If $\mathcal{M}$ is the family of such isomorphism classes, we see that the intersection $U_{M}=\bigcap_{T \in \mathcal{M}} U_{T}$ where $U_{T}$ is the open subset of $S$ coming from Proposition \ref{prop:Base_change_for_mixed_torsors} is a \textbf{very general} (the complement of a countable union of closed subsets with empty interior) subset of $S$. As $k$ is uncountable, we can find rational points within and thus there exists $m \in U_{M}(k)$ such that any mixed Nori-reduced torsor satisfies the base change condition at $m$ and the maximal Nori-reduced sub-torsor of any pull-back of a mixed torsor over $X_{m}$ is a torsor over the image $\ker(\pi^{N}(f)) \rightarrow \pi_{1}^{N}(X,x) \rightarrow G$ where $G$ is the group-scheme associated to a mixed torsor $T$. \\
Finally, as pull-back torsors satisfy the base change condition at any $s \in S(k)$ in this case from Proposition \ref{prop:Base_change_for_pull-backs}, by choosing a rational point $s \in U_{P} \cap U_{M}$ we have for $x \in X(k)$ over $s$, that the sequence
$$ \pi_{1}^{N}(X_{s},x) \rightarrow \pi_{1}^{N}(X,x) \rightarrow \pi_{1}^{N}(S,s) \rightarrow 1  $$
is exact as we wanted, because we have chosen $s$ such that we satisfy one Zhang's equivalent conditions for the homotopy exact sequence to hold, more specifically the one stated in Theorem \ref{teo:Homotopy_exact_sequence_conditions} (2).
\endproof
\proof[Proof of Theorem \ref{theorem:Main_theorem}]
As $\ker{\pi^{N}(f)}$ is finite, there exist a Nori-reduced torsor $S_{i} \rightarrow S$ and $X_{i} \rightarrow X$ its pull-back torsor, such that the finite torsor $\hat{X} \rightarrow X^{*}$ descends to a torsor $\hat{X}_{i} \rightarrow X_{i}$, pure with respect to $f_{i}:X_{i} \rightarrow S_{i}$. These schemes satisfy the hypotheses of Lemma \ref{lemma:Homotopy_exact_sequence_when_the_kernels_descends_to_the_base} and thus for a rational point $s^{\prime} \in S_{i}(k)$ with have an exact sequence
$$ \pi_{1}^{N}(X_{i,s^{\prime}},x^{\prime}) \rightarrow \pi_{1}(X_{i},x^{\prime}) \rightarrow \pi_{1}^{N}(S_{i},s^{\prime}) \rightarrow 1$$
where $X_{i,s^{\prime}}$ is the fiber of $f_{i}$ over $s^{\prime}$. Let $s,x$ be the images of the points $s^{\prime}$ and $x^{\prime}$ to $X$ and $S$ respectively, we see that we have a commutative diagram
$$\xymatrix{X_{i,s^{\prime}} \ar[r] \ar[d] &X_{i} \ar[d] \ar[r]^{f_{i}} & S_{i} \ar[d] \\
	X_{s} \ar[r] &X \ar[r]^{f} &S}. $$
We can easily see that $X_{i,s^{\prime}} \rightarrow X_{i}$ is a pointed Nori-reduced torsor and thus by taking the FGS of all the schemes involved we obtain the following commutative diagram of group-schemes
$$ \xymatrix{\pi_{1}^{N}( X_{i,s^{\prime}},x^{\prime}) \ar[r] \ar[d] & \pi_{1}^{N}(X_{i},x^{\prime}) \ar[d] \ar[r]^{f_{i}} & \pi_{1}^{N}(S_{i},s^{\prime}) \ar[d] \\
	\pi_{1}^{N}(X_{s},x) \ar[r] & \pi_{1}^{N}(X,x) \ar[r]^{f} & \pi_{1}^{N}(S,s)} $$
where all the vertical arrows are closed immersions. From the proof of Proposition \ref{prop:Properties_of_the_universal_pull-back}, we see that $\ker(\pi^{N}(f))=\ker(\pi^{N}(f_{i}))$ and thus we have additionally the following commutative diagram
$$ \xymatrix{\pi_{1}^{N}( X_{i,s^{\prime}},x^{\prime}) \ar[dd] \ar@{->>}[rd] \ar[rr] & & \pi_{1}^{N}(X_{i},x^{\prime}) \ar[dd] \\
	& \ker(\pi^{N}(f)) \ar[ur] \ar[dr] & \\
	\pi_{1}^{N}(X_{s},x)\ar[rr] \ar[ru] & & \pi_{1}^{N}(X,x)}. $$
From the homotopy exact sequence, the arrow $\pi_{1}^{N}( X_{i,s^{\prime}},x^{\prime}) \rightarrow \ker(\pi^{N}(f))$ is faithfully flat and as the arrow $\pi_{1}^{N}( X_{i,s^{\prime}},x^{\prime}) \rightarrow \pi_{1}^{N}( X_{s},x)$ is a closed immersion, we conclude that $\pi_{1}^{N}( X_{s},x) \rightarrow \ker(\pi^{N}(f))$ is faithfully flat, finishing the proof.
\endproof
\section{Applications to curve-connected varieties} \label{sec:Applications_to_curve_connectedness}
In this chapter we will apply the results of previous chapters to some particular cases of varieties connected by curves, or curve-connected. We will borrow most results and definitions related to these varieties from \cite{GounelasFreeCurves}. \\
One of the few examples of varieties with known fundamental group-schemes, besides abelian varieties, are rationally connected varieties.
\begin{defi}[Definition 3.1 \cite{GounelasFreeCurves}] \label{defi:Varieties_connected_by_genus_g_curves}
	Let $X$ be a variety over $k$, we will not assume $X$ is proper. We say for that $X$ is \textbf{connected by curves of genus g (resp. chain connected)} for some $g \geq 0$, if there exist a proper and flat family of curves $\mathcal{C} \rightarrow Y$ where $Y$ is a variety, whose geometric fibers are proper smooth irreducible curves of genus $g$ (resp. connected curves with smooth irreducible components that are curves of genus $g$), such that there exists a morphism $u:\mathcal{C} \rightarrow X$, making $u^{(2)}:\mathcal{C} \times_{Y} \mathcal{C} \rightarrow X \times_{k} X$ dominant. \\
	Moreover, if $X$ is connected (resp. chain connected) by curves of genus $g$ and $u^{(2)}$ is smooth at the generic point, we say that $X$ is \textbf{separably connected (resp. chain connected) by curves of genus g}.
\end{defi}
This is a generalization of the Definition \ref{defi:Rationally_connected_varieties} which is of course the latter definition with $g=0$. \\
The observation of Remark \ref{remark:Remarks_rationally_connected_varieties} also holds in general, meaning that in characteristic zero, varieties connected by curves of genus $g$ ($g \geq 0$) are separably connected by curves of genus g. In positive characteristic, the latter condition is stronger than the former. \\
Another notion of curve-connected varieties is the following:
\begin{defi}[Definition 3.3 \cite{GounelasFreeCurves}]
	Let $X$ be a variety over $k$, we will not assume $X$ is proper. Let $C$ be a curve, we say that $X$ is \textbf{$C$-connected} if there exist a variety $Y$ and a morphism $Y \times_{k} C \rightarrow X$, such that the induced map $u^{(2)}:\mathcal{C} \times_{k} \mathcal{C} \times_{k} Y \rightarrow X \times_{k} X$ dominant. \\
	Moreover, $u^{(2)}$ is smooth at the generic point, we say that $X$ is \textbf{separably $C$-connected}.
\end{defi}
The following proposition gives a special family of $C$-connected varieties.
\begin{proposition}[Proposition 3.5 \cite{GounelasFreeCurves}] \label{prop:A_special_type_of_C-connected_varieties}
	Let $X$ be a projective and smooth variety over an algebraically closed field and let $f:X \rightarrow C$ be a flat morphism to a smooth and projective curve whose geometric generic fiber is separably rationally connected. Then $X$ is $C$-connected.
\end{proposition}
We can describe the FGS of these varieties, by generalizing over positive characteristic, a result of Kollár \cite[Theorem 5.2]{Kollar1993} about the topological fundamental group of a rationally connected fibration of varieties over $\mathbb{C}$. We recall that the FGS of a smooth separably rationally connected variety is trivial (Proposition \ref{prop:FGS_of_rationally_connected_varieties}).
\begin{proposition} \label{prop:FGS_of_a_fibration_with_separably_rationally_connected_geometric_generic_fiber}
	Let $f:X \rightarrow Y$ be a faithfully flat fibration between proper, reduced and connected varieties over an algebraically closed field $k$, such that $\pi^{N}(f):\pi_{1}^{N}(X,x) \rightarrow \pi_{1}^{N}(Y,y)$ is faithfully flat for compatible rational points. We further assume that $X$ is smooth, $Y$ is integral, all geometric fibers are reduced and connected, and the geometric generic fiber is separably rationally connected. Then, $\pi^{N}(f)$ is an isomorphism.
\end{proposition}
\proof
If $\eta$ is the generic point of $Y$, we have that $\pi_{1}^{N}(X_{\bar{\eta}})$ is trivial as it is separably rationally connected. This implies that any essentially finite bundle over $X$ is the pull-back of an essentially finite bundle over $Y$ by Lemma \ref{lemma:Essentially_finite_bundles_trivial_over_the_geometric_generic_fiber_are_pull-backs}. Thus, $\pi^{N}(f)$ is also a closed immersion by \cite[2.21 (b)]{DeligneMilneTann}.
\endproof
We remark that in more generality, it suffices that $\pi_{1}^{N}(X_{\bar{\eta}})$ is trivial to obtain an isomorphism. We can apply this theorem to a our particular $C$-connected fibered varieties:
\begin{corollary} \label{corollary:FGS_of_particular_family_of_C-connected_varieties}
	Keeping the hypotheses of Proposition \ref{prop:A_special_type_of_C-connected_varieties}, the morphism $f:X \rightarrow C$ making $X$ a $C$-connected variety induces an isomorphism of fundamental group-schemes.
\end{corollary}
\proof
We can easily see that $\pi^{N}(f):\pi_{1}^{N}(X) \rightarrow \pi_{1}^{N}(C)$ is faithfully flat from a corollary under \cite[II Prop. 6 (c)]{NoriTFGS81}, thus this fibration satisfies the hypotheses of Proposition \ref{prop:FGS_of_a_fibration_with_separably_rationally_connected_geometric_generic_fiber}.
\endproof
In his article, Gounelas also characterizes \textbf{elliptically connected varieties} (Definition \ref{defi:Varieties_connected_by_genus_g_curves} with $g=1$) over an algebraically closed field of characteristic 0:
\begin{proposition}[Theorem 6.2 \cite{GounelasFreeCurves}] \label{prop:Elliptically_connected_varieties_in_characteristic_0}
	Let $X$ be a smooth and projective variety over an algebraically closed field of characteristic 0. Then $X$ is elliptically connected if and only if it is either rationally connected or a rationally connected fibration over an elliptic curve.
\end{proposition} 
Under the hypotheses of this proposition, the fibration $f:X \rightarrow E$ has separably rationally connected geometric fibers, then $X$ is $E$-connected by Proposition \ref{prop:A_special_type_of_C-connected_varieties}. \\
In positive characteristic, as we mentioned under Definition \ref{defi:Varieties_connected_by_genus_g_curves} and in Remark \ref{remark:Remarks_rationally_connected_varieties}, the condition of separability for rationally connected varieties is not always satisfied, and thus it is not certain if the fibration $f:X \rightarrow E$ makes $X$ elliptically connected, but at least we can characterize its FGS:
\begin{theorem} \label{teo:Rationally_connected_fibrations}
	Let $k$ be an uncountable algebraically closed field, and let $X$ be a smooth projective variety over $k$. Assume there is a projective fibration $f:X \rightarrow E$ where $E$ is an elliptic curve such that all geometric fibers are rationally connected. Then, there exists rational compatible points $x \in X(k)$ and $e \in E(k)$ such that the following sequence of group-schemes is exact:
	$$\pi_{1}^{N}(X_{s},x) \rightarrow \pi_{1}^{N}(X,x) \rightarrow \pi_{1}^{N}(E,e) \rightarrow 1 .$$
	Moreover, if $\text{char}(k)=0$, $f$ induces an isomorphism $\pi_{1}^{\text{\'et}}(X,\bar{x}) \rightarrow \pi_{1}^{\text{\'et}}(E,\bar{e})$ for compatible geometric points.
\end{theorem}
\proof
The first part comes directly from Theorem \ref{theorem:Main_theorem}, as the FGS of the geometric fibers of this morphism, that are rationally connected, are finite (Proposition \ref{prop:FGS_of_rationally_connected_varieties}). \\
If $k$ has characteristic zero, as the fundamental group-scheme and the \'etale fundamental group are effectively the same in this case \cite[Corollary 6.7.20]{Sza}, the conclusion follows from Corollary \ref{corollary:FGS_of_particular_family_of_C-connected_varieties}.
\endproof

\bibliography{biblio_papers}



\end{document}